\documentclass[12pt,leqno]{amsart}
\usepackage{amssymb,amsfonts,amsmath,amsopn,amstext,amscd,latexsym,hyperref,mathrsfs,verbatim}
\usepackage{epsfig}

\allowdisplaybreaks[1]

\markleft{On some properties of the functors... }
\theoremstyle{remark}
\newtheorem{para}{\bf}[subsection]

\newtheorem{rmk}[para]{\bf Remark}

\theoremstyle{definition}
\newtheorem{exam}[para]{\bf Example}

\newtheorem{dfn}[para]{\bf Definition}

\theoremstyle{plain}
\newtheorem{thm}[para]{\bf Theorem}
\newtheorem{lemma}[para]{\bf Lemma}

\newtheorem{cor}[para]{\bf Corollary}
\newtheorem{prop}[para]{\bf Proposition}

\newenvironment{numequation}{\addtocounter{para}{1}
\begin{equation}}{\end{equation}}

\newcommand{\bB}{{\bf B}}

\newcommand{\bG}{{\bf G}}

\newcommand{\bL}{{\bf L}}

\newcommand{\bP}{{\bf P}}

\newcommand{\bT}{{\bf T}}

\newcommand{\bbG}{{\mathbb G}}

\newcommand{\bbQ}{{\mathbb Q}}

\newcommand{\bbZ}{{\mathbb Z}}

\newcommand{\cF}{{\mathcal F}}
\newcommand{\cG}{{\mathcal G}}

\newcommand{\cI}{{\mathcal I}}

\newcommand{\cM}{{\mathcal M}}

\newcommand{\cO}{{\mathcal O}}

\newcommand{\frb}{{\mathfrak b}}

\newcommand{\frd}{{\mathfrak d}}

\newcommand{\frg}{{\mathfrak g}}

\newcommand{\frl}{{\mathfrak l}}

\newcommand{\frp}{{\mathfrak p}}
\newcommand{\frq}{{\mathfrak q}}

\newcommand{\frt}{{\mathfrak t}}
\newcommand{\fru}{{\mathfrak u}}

\newcommand{\frx}{{\mathfrak x}}

\newcommand{\frz}{{\mathfrak z}}

\newcommand{\uL}{{\underline L}}

\newcommand{\uN}{{\underline N}}

\newcommand{\Ad}{{\rm Ad}}

\newcommand{\diag}{{\rm diag}}

\newcommand{\Ext}{{\rm Ext}}
\newcommand{\GL}{{\rm GL}}
\newcommand{\Hom}{{\rm Hom}}
\newcommand{\hra}{\hookrightarrow}
\newcommand{\id}{{\rm id}}

\newcommand{\ind}{{\rm ind}}
\newcommand{\Ind}{{\rm Ind}}
\newcommand{\Lie}{{\rm Lie}}
\newcommand{\lra}{\longrightarrow}
\newcommand{\midc}{\;|\;}

\newcommand{\ob}{{\rm ob}}
\newcommand{\Qp}{{\bbQ}_p}
\newcommand{\ra}{\rightarrow}
\newcommand{\Rep}{{\rm Rep}}

\newcommand{\sub}{\subset}

\newcommand{\alg}{{\rm alg}}

\renewcommand{\qed}{{\hfill{\space} $\Box$}}

\begin{document}

\title[On some properties of the functors $\cF^G_P$]{On some properties of 
the functors $\cF^G_P$ from Lie algebra to locally analytic representations}
\author{Sascha Orlik}
\address{Fakult\"at f\"ur Mathematik und Naturwissenschaften, Bergische Universit\"at Wuppertal,
Gau{\ss}stra\ss{}e 20, D-42119 Wuppertal, Germany}
\email{orlik@uni-wuppertal.de}

\begin{abstract}
For a split reductive group $G$ over a finite extension $L$ of $\Qp$, and a 
parabolic subgroup $P \sub G$ we examine certain properties of the functors 
$\cF^G_P$ introduced  in \cite{OS2}. We discuss the aspects of faithfulness, 
projective and injective objects, Ext-groups and some kind of adjunction 
formula.
\end{abstract}

\maketitle

\tableofcontents

\section{Introduction}

This paper is a continuation of the work done in \cite{OS2}. In loc.cit. we 
constructed locally analytic representations in $K$-vector spaces of a 
$p$-adic reductive Lie group $G$ by introducing certain bi-functors 
$\cF^G_P:\cO^\frp_\alg \times \Rep^{\infty,a}_K(L_P)\to \Rep_K^{\rm loc. an.}(G)$. 
Here $P \sub G$ is a parabolic subgroup, $\frp = \Lie(P)$ its Lie algebra, 
and $\cO^\frp_\alg$ is a subcategory of the BGG-category\footnote{Deviating 
from the classical situation of modules over the enveloping algebra of a 
complex semisimple Lie algebra, we introduced in \cite{OS2} a version of 
this category which consists of modules over the enveloping algebra $U(\frg \otimes_L K)$.} $\cO^\frp$. 
Furthermore, $\Rep^{\infty,a}_K(L_P)$  is the 
category of smooth admissible representations of the Levi group $L_P$. We 
proved among others that these functors are exact in both arguments and gave 
a criterion for the irreducibility of those objects lying in the image of $\cF^G_P$. Using 
these properties one can derive a Jordan-H\"older series of any locally 
analytic representation $\cF^G_P(M,V)$ from the corresponding series of $M$ 
and $V$. 

\vskip8pt

In this paper we want to concentrate on properties of these functors for a 
split group $G$. We shall show that they behave fully faithful if the 
objects of the category $\cO^\frp$ are integral (i.e., they are contained
in the subcategory $\cO^\frp_\alg$ of modules such that all non-zero weight 
spaces belong to integral weights)  or generalized Verma modules. This
aspect has been considered by Morita in the case of $G={\rm SL}_2$, 
cf. \cite{Mo1,Mo2,Mo3,Mo4} and F{\'e}aux de Lacroix \cite{F}.
Concretely, we shall show:
\vskip8pt

{\bf Theorem 1.} {\it  For any $M_1, M_2 \in \cO^\frp_\alg$ the canonical  map
\begin{eqnarray*} \Hom_{\cO^\frp_\alg}(M_1,M_2) & \to &  \Hom_G(\cF^G_P(M_2),\cF^G_P(M_1)) \\
 f & \mapsto &  \cF^G_P(f) 
\end{eqnarray*}
is bijective (where $\cF^G_P(M):=\cF^G_P(M,{\bf 1})$ for the trivial $L_P$-representation ${\bf 1}$ ). }

\vskip8pt

To prove this statement we make use of the (naive) topological Jacquet functor of locally analytic representations
and more generally of an analogue of the Casselman-Jacquet functor $\cG^G_P:U \mapsto \varinjlim_k H^0(\fru_P^k,U')$
which behaves almost like a section for $\cF^G_P$. This 
topic  is a continuation of the theory  started  by Schraen and the author respectively  Breuil  \cite{OSch, Br}. 

\vskip8pt
By the above theorem
we can characterize projective and injective objects which lie in the essential image $\cF^P_\alg$ of the 
functor $\cF^G_P:\cO^\frp_\alg \to \Rep_K^{\rm loc. an.}(G)$. More precisely, it follows that  $M\in \cO^{\frp}_\alg$ is projective (resp. injective) as an object in 
$\cO^\frp$  if and only if $\cF^G_P(M)$ is injective
(resp. projective) in $\cF^P_\alg.$ Hence if we denote for a given integer $i\geq 0$, by $\Ext^i_{\cF^P_\alg}$ the corresponding Ext-group 
then the  natural morphism
\begin{eqnarray*} \Ext^i_{\cO^\frp_\alg}(M_1,M_2) & \to &  \Ext^i_{\cF^P_\alg}(\cF^G_P(M_2),\cF^G_P(M_1))
 \end{eqnarray*}
 is bijective. These Ext-groups are of course different from those considered more generally in the category of locally analytic 
$G$-representations, cf. \cite{K2}. These can be seen as an analogue of relating the groups $\Ext^i_{\frg}(M_1,M_2)$
and $\Ext^i_{\cO}(M_1,M_2)$ for two objects $M_1,M_2 \in \cO.$

For considering also smooth contributions in this context, we extend $\cF^G_P$ to a bi-functor 
$\cF^G_P:\cO^\frp_\alg \times \Rep^{\infty,\infty}_K(L_P)\to \Rep_K^{\rm loc. an.}(G)$ where $\Rep^{\infty,\infty}_K(L_P)$
denotes the category of smooth $L_P$-representations of countable dimension. The latter object has enough injectives and 
projectives. We let $^\infty\overline{\cF^P}$ be the smallest abelian subcategory of $\Rep_K^{\rm loc. an.}(G)$ 
which contains the essential images of all bi-functors $\cF^G_Q$ with  $Q\supset P.$
It turns out that  $^\infty\overline{\cF^P}$ has enough injective and projective objects.
 More precisely, we deduce this fact from the following statement.
 
\vskip8pt 
 
{\bf Theorem 2:} {\it Let $M\in \cO^{\frp}_\alg$ be a projective (resp. injective)  object  and let $V$
be an injective (resp. projective) smooth $L_P$-representation of countable dimension. Then  $\cF^G_P(M,V)$ is injective (resp. projective) 
in $^\infty\overline{\cF^P}.$}
\vskip8pt

As an application we are able to determine
extensions of generalized Steinberg representations in the category $^\infty\overline{\cF^B}$. For a parabolic subgroup 
$P\sub G$ the associated representation is given by the quotient 
$V^G_P= \Ind^G_P(\bf 1) / \sum_{Q \supsetneq P} \Ind^G_Q(\bf 1)$ where $\Ind^G_P({\bf 1})$ is the locally
analytic induction with respect to the trivial $P$-representation ${\bf 1}.$ For a subset 
$I\sub \Delta$ of a fixed set of simple roots, let $P_I$ be the corresponding standard parabolic subgroup.
The next result has the same structure as in the smooth setting \cite{Da,O1}.

\vskip8pt

\noindent {\bf Theorem 3:} Let $G$ be semi-simple. {\it Let $I,J \sub \Delta$. Then

$$\Ext_{^\infty\overline{\cF^B}}^i(V^G_{P_I}, V^G_{P_J})=\left\{ \begin{array}{cc}
                                   K \;; & | I \cup J \setminus I\cap J| =i \\
                                  (0) \;; & \mbox{otherwise}
                                \end{array} \right. .
 $$}

 Finally we deduce from the naive Jacquet functor applied to different Borel subgroups lying in the same
 apartment an adjunction formula (in the sense of Bernstein). Let $U_B$ be the unipotent radical of a fixed Borel 
 subgroup $B.$ If we denote for a given $G$-representation $V$ by $V_{U_B}$ its (naive) topological Jacquet module
 then the map below is defined as follows: For an element $f$  of the LHS, the corresponding element on the RHS is
 given by the composition of the inclusion
 $((w_0\cdot_{\overline{B}}\chi)^{-1})^{w_0} \hookrightarrow \Ind^G_{\overline{B}}(\chi^{-1})_{U_B}$
 with the map $f_{U_B}: \Ind^G_{\overline{B}}(\chi^{-1})_{U_B} \to \cF^G_{B^w}(M)_{U_B}.$

 \vskip8pt
 
{\bf Theorem 4:} {\it  Let $\chi$ be a dominant algebraic character of $T$ and let $M \in \cO^{\frb^w}_\alg$ be a highest weight module. Then
 
 $$\Hom_G(\Ind^G_{\overline{B}}(\chi^{-1}),\cF^G_{B^w}(M))=
 \Hom_T(((w_0\cdot_{\overline{B}}\chi)^{-1})^{w_0},\cF^G_{B^w}(M)_{U_B}) \;.$$}

\vskip8pt

{\it Acknowledgments.} I am very grateful to Matthias Strauch for all his criticism and his  contributions to this paper since the very beginning. Here I also thank  Indiana University for hospitality where parts of the paper were created.  I am  greatly indebted to Florian Herzig who carefully read a previous version of this paper and pointed out several inaccurate statements and provided very helpful suggestions.
I thank Volodymyr Mazorchuk for his answer on a mathematical question concerning the category $\cO.$ Finally I thank Christophe Breuil for his steady interest in this paper.
This research was conducted in the framework of the research training group
\emph{GRK 2240: Algebro-Geometric Methods in Algebra, Arithmetic and Topology}, which is funded by the DFG. 
\vskip8pt

{\it Notation and conventions.} We denote by $p$ a prime number and consider fields $L \sub K$ which are both finite extensions of $\Qp$. 
Let $O_L$ and $O_K$ be the rings of integers of $L$, resp. $K$, and let $|\cdot |_K$ be the absolute value on $K$ such that $|p|_K = p^{-1}$. The field $L$ is our ''base field'', whereas we consider $K$ as our ''coefficient field''. For a locally convex $K$-vector space $V$ we denote by $V'_b$ its strong dual, i.e., the $K$-vector space of continuous linear forms equipped with the strong topology of bounded convergence. Sometimes, in particular when $V$ is finite-dimensional, we simplify notation and write $V'$ instead of $V'_b$. All finite-dimensional $K$-vector spaces are equipped with the unique Hausdorff locally convex topology.

\vskip8pt

We let $\bG_0$ be a split reductive group scheme over $O_L$ and $\bT_0 \sub \bB_0 \sub \bG_0$ a maximal split torus and a 
Borel subgroup scheme, respectively. We denote by $\bG$, $\bB$, $\bT$ the base change of $\bG_0$, $\bB_0$ and $\bT_0$ 
to $L$. By $G_0 = \bG_0(O_L)$, $B_0 = \bB_0(O_L)$, etc., and $G = \bG(L)$, $B = \bB(L)$, etc., we denote the 
corresponding groups of $O_L$-valued points and $L$-valued points, respectively. Standard parabolic subgroups of 
$\bG$ (resp. $G$) are those which contain $\bB$ (resp. $B$). For each standard parabolic subgroup $\bP$ (or $P$) we 
let $\bL_\bP$ (or $L_P$) be the unique Levi subgroup which contains $\bT$ (resp. $T$) and ${\bf U_P}$ (or $U_P$) its unipotent radical. 
Finally, Gothic letters $\frg$, $\frp$, etc., will denote the Lie algebras of $\bG$, $\bP$, etc.: 
$\frg = \Lie(\bG)$, $\frt = \Lie(\bT)$, $\frb = \Lie(\bB)$, $\frp = \Lie(\bP)$, $\frl_P = \Lie(\bL_\bP)$, etc.. 
Base change to $K$ is usually denoted by the subscript ${}_K$, for instance, $\frg_K = \frg \otimes_L K$. 

\vskip8pt

We make the general convention that we denote by $U(\frg)$, $U(\frp)$, etc., 
the corresponding enveloping algebras, {\it after base change to K}, i.e., 
what would be usually denoted by 
$U(\frg) \otimes_L K$, $U(\frp) \otimes_L K$, and so on. All distribution 
algebras appearing in this paper are tacitly assumed to be distribution 
algebras with coefficient field $K$, and we write $D(H)$ for the 
distribution algebra $D(H,K)$.
For a real number $r<1$ with $r\in p^{\bbQ}$, and $H$ compact we let $D_r(H)$ be the Banach space completion in 
the sense of \cite{ST2} so that $D(H)=\varprojlim_r D_r(H).$
Finally, $\Rep^{\rm loc. an.}_K(G)$ denotes 
the category of locally analytic representations of $G$ on barreled locally 
convex Hausdorff $K$-vector spaces.

\vspace{0.5cm}

\section{A review of the categories \texorpdfstring{$\cO^\frp_\alg$}{} and the functors \texorpdfstring{$\cF^G_P$}{}}

For the convenience of the reader we recall here the definitions of the categories $\cO^\frp_\alg$, as well the functors $\cF^G_P$, and state some of the key results about those, cf. \cite{OS2}.

\subsection{The categories \texorpdfstring{$\cO^\frp$}{} and \texorpdfstring{$\cO^\frp_\alg$}{}}

For a parabolic subalgebra $\frp \sub \frg$ we define $\cO^\frp$ to be the full subcategory of the category of all $U(\frg)$-modules consisting of objects $M$ which possess the following properties:

\begin{enumerate}
\item $M$ is a finitely generated $U(\frg)$-module.

\item $M$ decomposes as a direct sum of one-dimensional $\frt_K$-modules.

\item The action of $\frp$ on $M$ is locally finite, i.e., for every $m \in M$ the $K$-subspace $U(\frp).m \sub M$ is finite-dimensional. 
\end{enumerate}

\medskip
We also put $\cO = \cO^\frb$. Given $M \in \ob(\cO)$ and $\lambda \in \frt_K^*$ we set 
$$M_\lambda = \{m \in M \midc \forall \frx \in \frt: \frx.m = \lambda(\frx)m \} \;.$$  
We call $\lambda \in \frt_K^*$ {\it algebraic} if it is in the image of the 
canonical homomorphism 
$$X^*(\bT) = \Hom_{\rm alg. \hskip2pt gps}(\bT,\bbG_{m,L}) \lra \frt_K^* \;, \;\; \chi \mapsto {\rm d}\chi \;.$$

We define $\cO^\frp_\alg$ to be the full subcategory of $\cO^\frp$ which consists of $M \in \ob(\cO^\frp)$ such that $M_\lambda \neq 0$ implies that $\lambda$ is algebraic. 

\subsection{The functors \texorpdfstring{$\cF^G_P$}{}} \setcounter{enumi}{0}

We consider an object $M$ of the category $\cO^\frp_\alg$. By \cite[3.2]{OS2}, the locally finite action of $\frp$ on $M$ lifts canonically to a locally finite algebraic action of the algebraic group $\bP_K$. Since $M$ is finitely generated over $U(\frg)$, we can choose a $\bP_K$-subrepresentation $W \sub M$ which generates $M$ as a $U(\frg)$-module. Thus we get an exact sequence
\begin{numequation}\label{display-frd} 0 \lra \frd \lra U(\frg) \otimes_{U(\frp)} W \lra M \lra 0 \;
\end{numequation}
Let $\Ind^G_P(W')$ be the locally analytic induction of the dual space $W'.$  There is a pairing

$$
\begin{array}{rccc}
\langle \cdot , \cdot \rangle_{C^{an}(G,K)}: & \left(D(G) \otimes_{D(P)} W \right) \otimes_K \Ind^G_P(W') & \lra & C^{an}(G,K) \\
&&&\\
& (\delta \otimes w) \otimes f & \mapsto & \Big[ g \mapsto \delta(x \mapsto f(gx)(w))\Big]
\end{array}$$
which extends for any smooth admissible $L_P$ representation, to a pairing

\begin{numequation}\label{display-pairing} \langle \cdot , \cdot \rangle_{C^{an}(G,V)}: \left(D(G) \otimes_{D(P)} (W\otimes_K V') \right) 
\otimes_K \Ind^G_P(W'\otimes_K V)  \lra  C^{an}(G,K) \;.
\end{numequation}

Here and in the following we always equip an admissible smooth representation $V$ with the finest locally convex topology (the final topology with respect to which all inclusion maps $V_1 \hra V$, where $V_1 \sub V$ is finite-dimensional, are continuous, cf. \cite[ch. I, \S 5, E]{S1}). As $W$ is finite-dimensional, the topology of the inductive tensor product $W \otimes_{K,\iota} V'$ coincides with that of the projective tensor product $W \otimes_{K,\pi} V'$, and we thus write simply $W \otimes_K V'$ for this topological vector space (cf. \cite[\S 17]{S1} for tensor products of topological vector spaces and the notation used here). We set
$$\Ind^G_P(W' \otimes_K V)^\frd = \{f \in \Ind^G_P(W' \otimes_K V) \midc \forall  \frz \in \frd: \langle \frz, f \rangle_{C^{an}(G,V)} = 0 \} \;.$$
Then $\Ind^G_P(W' \otimes_K V)$ carries the structure of a $K$-vector space of compact type, and it is easily seen that the linear maps $\Ind^G_P(W' \otimes_K V) \ra C^{an}(G,K)$ defined by the pairing (\ref{display-pairing}) are continuous. Hence $\Ind^G_P(W' \otimes_K V)^\frd$ is a closed subspace of $\Ind^G_P(W' \otimes_K V)$, and we equip it with its subspace topology. As such it is again of compact type \cite[1.2]{ST1}. By \cite[2.1.2]{Em1}, $\Ind^G_P(W' \otimes_K V)$ is an admissible locally analytic representation of $G$, and $\Ind^G_P(W' \otimes_K V)^\frd$ is thus again an admissible locally analytic representation of $G$, cf. \cite[6.4]{ST2}.

\vskip8pt
If $W_1 \sub W_2$ are two finite-dimensional $\frp_K$-stable subspaces which generate $M$ as a $U(\frg)$-module, then there is a canonical continuous morphism of $G$-representations 
\begin{numequation}\label{display-independence_W} \Ind^G_P(W_2' \otimes_K V)^{\frd_2} \lra \Ind^G_P(W_1' \otimes_K V)^{\frd_1} \;,
\end{numequation}
where $\frd_i$ is defined by the corresponding exact sequence \ref{display-frd}. By \cite[4.5]{OS2} the map \ref{display-independence_W} is actually an isomorphism of topological vector spaces. We thus see that the formation of $\Ind^G_P(W' \otimes_K V)^\frd$ is independent of the choice of $W$, and we put
\begin{numequation}\label{display-defF}
 \begin{array}{rcl} \cF^G_P(M,V) & = & \Ind^G_P(W' \otimes_K V)^\frd \;.
\end{array}
\end{numequation}

Denote by $\Rep^{\infty,a}_K(L_P)$ the category of smooth admissible representations of $L_P$ and by $\Rep^{\rm loc. an.}_K(G)$ the category of locally analytic representations of $G$ on $K$-vector spaces. Then we have a   bi-functor
\begin{numequation}\label{display-intro_FGP} \cF^G_P: \cO^\frp_\alg \times \Rep^{\infty,a}_K(L_P) \lra \Rep^{\rm loc. an.}_K(G) \;.
\end{numequation}
If $V={\bf 1}$ denotes the trivial representation, then we simply write $\cF^G_P(M)$ for $\cF^G_P(M,V)$. 

\vskip8pt
\begin{thm}\label{theorem_mainOS2} 

{\rm \cite[4.9, 5.3]{OS2}}

a) The bi-functor $\cF^G_P$ is exact in both arguments.

\vskip8pt

b) (PQ-formula) If $Q \supset P$ is a parabolic subgroup, $\frq = \Lie(Q)$, and $M$ an object of $\cO^\frq_\alg$, then

$$\cF^G_P(M,V) = \cF^G_Q(M,i^{L_Q}_{L_P(L_Q \cap U_P)}(V)) \;,$$

\vskip8pt

where $i^{L_Q}_{L_P(L_Q \cap U_P)}(V)=i^Q_P(V)=\ind^Q_P(V)$ denotes the corresponding induced representation in the category of smooth representations.

\vskip8pt

c) Suppose $M \in \cO^\frp_\alg$ is simple and that 
$\frp$ is maximal for $M$ (i.e., if $M \in \cO^\frq$ with $\frq \supset \frp$, then $\frq = \frp$). Let $V$ be a smooth and irreducible $L_P$-representation. Then $\cF^G_P(M,V)$ is topologically irreducible as a $G$-representation\footnote{Here we assume that if the root system 
$\Phi = \Phi(\frg,\frt)$ has irreducible components of type $B$, $C$ or 
$F_4$, then $p > 2$, and if $\Phi$ has irreducible components of type $G_2$, 
we assume that $p > 3$.}.
\end{thm}

\subsection{A description of the dual space of \texorpdfstring{$\cF^G_P(M)$}{}}

By \cite[3.2]{ST1}, $M$ carries a canonical structure of a module over the locally analytic distribution algebra $D(P) = D(P,K)$. Let $D(\frg,P)$ be the subring of $D(G)$ generated by $U(\frg)$ and $D(P)$ inside $D(G)$. By \cite[3.6]{OS2}, the $U(\frg)$-module structure on $M$ and the $D(P)$-module structure on $M$ agree on the subring $U(\frp)$, and there is a unique structure of a module over $D(\frg,P)$ on $M$ which extends these module structures. By \cite[3.7]{OS2} there is a canonical isomorphism of $D(G)$-modules
\begin{numequation} \cF^G_P(M)' \; \cong \; D(G) \otimes_{D(\frg,P)} M  \;.
\end{numequation}

\vspace{0.5cm}

\section{Jacquet functors} 

\subsection{Jacquet module for irreducible objects $\cF^G_P(M,V)$}

The first part of this section deals with a  r\'esum\'e of results 
formulated in \cite{OSch,Br}, where the Jacquet functor of simple objects 
$\cF^G_P(M,V)$ with $M\in \cO^\frp_\alg$ was discussed. 

\vskip8pt

Let $P$ be a parabolic subgroup of $G$ with Levi decomposition $P=L_PU_P.$
For a locally analytic $P$-representation $V$, let $V(U_P)$ be the subspace generated 
 by the expressions $uv-v,$ with $u\in U_P, v\in V$ and let $\overline{V(U_P)}$ be
 its topological closure which is a $P$-stable subspace of $V$. Denote by 
$$\overline{H}_0(U_P,V):=V_{U_P}:=V/\overline{V(U_P)}$$ the corresponding quotient 
(the naive topological Jacquet module). It is the
largest Hausdorff quotient of $V$ on which $U_P$ acts trivially.

\begin{lemma}
 The space $\overline{H}_0(U_P,V)$ has  the canonical structure of a locally analytic $P$-representation.
\end{lemma}

\begin{proof} Since $\overline{V(U_P)}$  is a closed subspace of $V$, the quotient is a barreled locally convex Hausdorff vector space.
 Moreover the orbit maps $P \to \overline{H}_0(U_P,V)$  are clearly locally analytic since these are induced
 by the locally analytic orbit maps $P \to V.$
\end{proof}

On the other hand, if $V$ is of compact type then its dual $V'$ is a $K$-Fr\'echet space
equipped with a continuous and locally analytic action of $P$. We let $H^0(U_P,V')$ be the subspace of $V'$ consisting of vectors which are fixed by $U_P.$ This is a closed subspace so that
$H^0(U_P,V')$ inherits the structure of a $K$-Fr\'echet space equipped with an action of $P$, as well.
Since the action of $P$ is locally analytic we can define the subspace 
$$H^0(\fru_P,V')=\{w \in V' \mid x\cdot w=0, \forall x \in \fru_P\}$$ and 
the Hausdorff quotient $\overline{H}_0(\fru_P,V)=V/\overline{\fru_P V}$, as well. Then  $H^0(\fru_P,V')$ is by the continuity of the $\frp$-action a closed $P$-equivariant subspace  of $V'$ with
$H^0(U_P,V')\subset H^0(\fru_P,V') .$

\begin{lemma}\label{duality}
Let $V$ be of compact type.  Under the duality pairing  $V \times V' \to K$ the subspace $H^0(U_P,V')$  (resp. $H^0(\fru_P,V')$) is the topological dual of
 $\overline{H}_0(U_P,V)$ (resp. $\overline{H}_0(\fru_P,V))$ as $P$-representations.
\end{lemma}

Let $Q$ be another  parabolic subgroup with $P\sub Q$ and let $Q=L_Q\cdot U_Q$ be its Levi decomposition. 
In this sequel we want to determine for certain objects $M\in \cO^{\frq}_\alg$ and smooth admissible $L_Q$-representations
$V$  the $L_P$-representations $H^0(U_P,\cF_Q^G(M,V)')$.

\vskip8pt
For a compact open subgroup $H\subset G_0$, let $P_H=H \cap P$ and let $D(P_H,\frg) \subset D(H)$ be the subring generated by
$\frg$ and $P_H$. Moreover, let $D_r(P_H,\frg)$ be the Banach space completion of $D(P_H,\frg)$  inside $D_r(H)$ and set $M_r:=D_r(P_H,\frg) \otimes_{D(P_H,\frg)} M.$
\begin{prop}\label{non-trivial_intersection}
Let $M$ be an object of $\cO^\frp_\alg$. We have an inclusion preserving bijection
\begin{eqnarray*}
\Big\{\mbox{closed $U(\frl_P)$-invariant subspaces of $M_r$} \Big\} &
\stackrel{\sim}{\longrightarrow} &
\Big\{\mbox{$U(\frl_P)$-invariant subspaces of $M$}\Big\}. \\
S & \longmapsto & S\cap M
\end{eqnarray*}
The inverse map is induced by taking the closure. 
\end{prop}

\begin{proof}
 By \cite{OSch} we have such an inclusion preserving bijection
\begin{eqnarray*}
\Big\{\mbox{closed $U(\frt)$-invariant subspaces of $M_r$} \Big\} &
\stackrel{\sim}{\longrightarrow} &
\Big\{\mbox{$U(\frt)$-invariant subspaces of $M$}\Big\}. \\
S & \longmapsto & S\cap M
\end{eqnarray*}
for $U(\frt)$-submodules.
But for a closed $U(\frt)$-submodule $N \sub M_r$ the intersection $N \cap M$ is $\frl_P$-stable if and only if $N$ is 
$U(\frl_P)$-stable. Indeed, whereas one direction is obvious the other one follows by density arguments. The claim follows.
\end{proof}

Recall that if $M$ is a Lie algebra representation of $\frg,$
then $H^0(\fru_Q,M)=\{m\in M \mid \frx\cdot m=0 \;  \forall\; \frx \in \fru_Q\}$ denotes the subspace of vectors 
killed by $\fru_Q.$ 

\begin{cor} \label{ncohomology} Let $M$ be an object of $\cO^\frp_\alg.$ 
Then $H^0(\fru_P,M_r)=H^0(\fru_P, M).$ In particular, $H^0(\fru_P,M_r)$  is finite-dimensional.
\end{cor}

\begin{proof}
We clearly have $H^0(\fru_P,M_r) \cap M = H^0(\fru_P,M).$ As $H^0(\fru_P,M_r)$ is closed
in $M_r$ by the continuity of the action of $\frg$  and as $H^0(\fru_P,M)$ is finite-dimensional (!!!!)
and therefore complete the statement follows by Proposition 
\ref{non-trivial_intersection}.
\end{proof}

\begin{lemma}\label{Lemma_identity}
Let $M$ be an object of $\cO^\frq_\alg$ where $P\sub Q$  and let $V$ be a smooth admissible $L_Q$-representation.
Then there is an identification
$$H^0(\fru_P, \cF_Q^G(M,V)') = H^0(\fru_P, \cF_Q^G(M)') \hat{\otimes}_K V'$$ 
of Fr\'echet spaces with $P$-action (Here the action of $P$ on $V'$ is given by the composite $P\hookrightarrow Q \twoheadrightarrow L_Q.$) . 
\end{lemma}

\begin{proof}
The proof is the same as in \cite{OSch} by  replacing $\fru_B$ by $\fru_P$.
\end{proof}

For $M\in \cO^\frq_\alg$, let $W\sub M$ be a finite-dimensional algebraic $Q$-subrepresentation such that the
map (a morphism in $\cO^\frq_\alg$)
$$M(W):=U(\frg) \otimes_{U(\frq)} W \to M$$ is surjective. If $M$ is simple so that we may assume that $W$ comes via inflation from
an irreducible $L_Q$-representation then  
$H^0(\fru_Q,M)=W$. We set in this situation $W_M:=H^0(\fru_Q,M).$

\vspace{0.5cm}

Now we are able to state one of the building blocs of this paper which is an analogue of a statement by Hecht and Taylor  dealing with
representations of real Lie groups and Harish-Chandra modules \cite{HT1,HT2}. Its proof is already contained
in \cite{OSch, Br}, so that this result is not really new. Nevertheless, for later use we are going to give a proof of it.

\begin{thm}\label{thm_H0}
    Let $M$ be a simple object of $\cO^\frq_\alg$ with $Q$ maximal for $M$.  
    Let $V$ be a smooth admissible $L_Q$-representation. 
  Then for $P\sub Q$ there are $P$-equivariant topological isomorphisms
  
  \begin{equation*}
    H^0(U_P, \cF_Q^G(M,V)')=H^0(\fru_P,W_M) \otimes_K J_{U_P \cap L_Q}(V)',
  \end{equation*}
  and
    \begin{equation*}
      \overline{H}_0(U_P,\cF_Q^G(M,V))=H_0(\fru_P,W_M') \otimes_K J_{U_P \cap L_Q}(V),
  \end{equation*}
  where $J_{U_P \cap L_Q}$ is the usual Jacquet functor for the unipotent subgroup
  $U_P \cap L_Q \sub L_Q$.
\end{thm}

\begin{proof}
 By the duality treated  in Lemma \ref{duality} it suffices to check the first identity.
 Here we assume first that $V={\bf 1}$ is the trivial representation and that $P=Q$.
 Write $M=L(\lambda)$ for some algebraic character $\lambda$
 of $T.$
 
  We follow the proof of  \cite[Thm. 3.5]{OSch}. Let $\cI  \sub G$ be the standard Iwahori subgroup. 
 For $w\in W$, let $M^w=M$ be the $D(\frg, \cI \cap wP_0w^{-1})$-module with the twisted action given by conjugation with $w$.
  Let $I\sub \Delta$ be a subset with $P=P_I.$   
  The Bruhat decomposition $G_0=\coprod_{w \in W^I} \cI w P_0$ induces a
  decomposition 
  \begin{eqnarray*}
    D(G_0) \otimes_{D(\frg,P_0)} M  & \simeq &  \bigoplus_{w \in W^I}
     D(\cI) \otimes_{D(\frg, \cI \cap wP_0w^{-1})}  M^w \\
     & \simeq &  \bigoplus_{w \in W^I}  D(w^{-1}\cI w) \otimes_{D(\frg, w^{-1}\cI w \cap P_0)} M .
      \end{eqnarray*}
         For each $w \in W^I$, we have
  $$H^0(\fru_P, D( \cI ) \otimes_{D(\frg, I \cap wP_0w^{-1})} M^w)
    \simeq H^0(\mathrm{Ad}(w^{-1}) (\fru_P), D(w^{-1}\cI w) \otimes_{D(\frg,
      w^{-1}Iw \cap P_0)} \otimes M ).$$
  We can write each summand in the shape 
  $$\cM^w:=D(w^{-1}\cI w) \otimes_{D(\frg, w^{-1}\cI w \cap P_0)} M =
  \varprojlim_r \cM_r^w$$    where $\cM_r^w=D_r(w^{-1}\cI w) \otimes_{D(\frg, w^{-1}\cI w \cap P_0)} M$. If we denote
  by $M^w_r$ the topological closure of $M$ in $\cM^w_r$, we get by \cite[1.4.2]{K1} finitely many elements  $u \in U^-_{P_0}$ such that
    \begin{equation*} \cM_r^w \simeq \bigoplus_u \delta_u \otimes M_r^w
  \end{equation*}
  and the action of $\frx \in \mathrm{Ad}(w^{-1}) (\fru_P)$ is given by
  \begin{equation*}
    \frx \cdot \sum \delta_u \otimes m_u=\sum \delta_u \otimes
    \mathrm{Ad}(u^{-1}(\frx))m_u.
  \end{equation*}
      
In \cite[Thm 3.5]{OSch} it is explained that for $w\neq 1$, there is a non-trivial element $\frx \in \fru_P^- \cap
  \mathrm{Ad}(w^{-1})(\fru_P)$.  Since $P$ is maximal for $M$ and $M$ is simple we deduce by  \cite[Corollary 8.6]{OS3}, that elements
  of $\fru_P^-$ act injectively on $M$, and as explained in Step 1 of \cite[Theorem 5.7]{OS3} they act 
  injectively on
  $M_r^w$, as well. We conclude that $H^0(\mathrm{Ad}(u^{-1})(\mathrm{Ad}(w^{-1})(\fru_P)), M_r^w)=0$ for $w\neq 1$ 
  since $\mathrm{Ad}(u^{-1})(\frx) \in
  \fru_P^-$. So  $H^0(\mathrm{Ad}(u^{-1})(\mathrm{Ad}(w^{-1})(\fru_P)), \cM_r^w)=0$. Hence by passing to the limit we
  get $H^0(\mathrm{Ad}(w^{-1})(\fru_P) , \cM^w)=0$ for $w\neq 1$. 

  \vskip8pt
  Now consider the case $w=1$. Again we may write $D(\cI)_r=\bigoplus
  \delta_u D(\frg,P_0)_r$ for a finite number of $u \in U_{P,0}^-$, so that $D(\cI)_r
  \otimes_{D(\frg,P_0)_r} M^1_r=\bigoplus_u \delta_u \otimes M^1_r$. We shall show
  that if $u \notin U_{P,0}^- \cap D(\frg,P_0)_r$, then
  $H^0(\mathrm{Ad}(u^{-1}) \fru_P, M^1_r)=0$. Here we will use Step
  $2$ in the proof of \cite[Theorem 4.7]{OS3} where we have used that $P$ is maximal for $M.$
  Let $\hat{M}$ be the
  formal completion of $M$, i.e. $\hat{M}=\prod_{\mu} M_{\mu}$ which is a
  $\frg$-module. The action of $\fru_P^-$ can be
  extended to an action of $U_P^-$ as explained in loc.cit. If $\frx \in \frg$ and $u \in U_P^-$,
  the action of $\mathrm{ad}(u) \frx$ on $M_r$ is the restriction of the composite
  $u \circ \frx \circ u^{-1}$ on $\hat{M}$. 
  Let $\hat{M}$ be the
  formal completion of $M$, i.e. $\hat{M}=\prod_{\mu} M_{\mu}$ which is a
  $\frg$-module. The action of $\fru_P^-$ can be
  extended to an action of $U_P^-$ as explained in loc.cit. If $\frx \in \frg$ and $u \in U_P^-$,
  the action of $\mathrm{ad}(u) \frx$ on $M_r$ is the restriction of the composite
  $u \circ \frx \circ u^{-1}$ on $\hat{M}$. As a consequence, we get 
  \begin{eqnarray*}
   H^0(\mathrm{ad}(u^{-1})\fru_P, M^1_r) & = & M^1_r \cap u^{-1} \cdot H^0(\fru_P, \hat{M})\\
   & = & M^1_r \cap u^{-1}W_M
  \end{eqnarray*}
 since $H^0(\fru_P, \hat{M})=H^0(\fru_P, M)=W_M$ (Here and in the sequel we copy the argumentation of Breuil \cite{Br}). 
 Let $v^+$ be a highest weight vector of $M$.  If the term  $H^0(\mathrm{ad}(u^{-1}) \fru_P, M^1_r) \neq (0)$ does not 
 vanish, then we have consequently $u^{-1}W_M \cap M^1_r\neq (0)$. 
 we deduce that $u^{-1}W_M \sub M^1_r$ since $W_M$ is irreducible.
 In particular $u^{-1}v^+ \in M_r$.
  By the proof of \cite[Theorem 4.7]{OS3}, this
  gives a contradiction if $u \notin U_P^- \cap D_r(\frg, P_0)$. 
  Hence by passing to the limit and using  Corollary \ref{ncohomology} we obtain finally an isomorphism of Fr\'echet spaces
  \begin{equation*}H^0(\fru_P, D( \cI ) \otimes_{D(\frg,P_0)} M )
    \simeq H^0(\fru_P, M)=W_M.
\end{equation*}

  \vskip8pt
 Next we consider the general situation where also a smooth representation is involved and where $P\sub Q.$ 
 By Lemma \ref{Lemma_identity} and what we proved above we get 
    \begin{equation*}
    H^0(\fru_P, \cF_Q^G(M,V)')=H^0(\fru_P,H^0(\fru_Q, \cF_Q^G(M,V)')=H^0(\fru_P,W_M) \otimes_K V',
  \end{equation*}
 
 Since $H^0(U_P, \cF_Q^G(M,V)')$ is a subspace of 
 $H^0(\fru_P, \cF_Q^G(M,V)')$ the latter one is stable by the 
  action of $U_P.$   Thus we deduce by Lemma \ref{Lemma_identity} that
  \begin{eqnarray*}
   H^0(U_P, \cF_Q^G(M,V)') & = &  H^0(U_P,H^0(\fru_P, \cF_Q^G(M,V)')) \\
	               & = &  H^0(U_P,H^0(\fru_P, W_M) \otimes_K V') \\
			 & = & H^0(\fru_P, W_M) \otimes_K J_{U_P\cap L_Q}(V)'.
\end{eqnarray*}
The last identity follows from the fact that the action of $U_P$ on $W_M$
is induced by the one of $\fru_P.$ 
 \end{proof}

\subsection{An analogue of the Casselman-Jacquet functor}

The next result generalizes Theorem \ref{thm_H0} by considering non-necessarily  simple modules $M$.

\begin{prop}\label{thm_H0_general}
  Let $M\in \cO^\frp_\alg$ and let $V$ be a smooth admissible  $L_P$-representation. Then there are canonical (topological)
  identities
$$H^0(\fru_P,\cF^G_P(M,V)')=\bigoplus_{W\sub H^0(\fru_P,M)} W \otimes S_W(V)'$$ 
  and 
  $$\overline{H}_0(\fru_P,\cF^G_P(M,V))=\bigoplus_{W\sub H^0(\fru_P,M)} W' \otimes S_W(V)$$ 
  of $P$-representations where 
  $S_W(V)$ is a  quotient of $i^{P_W}_P(V)_{|P}$ for some standard parabolic subgroup $P_W\supset P$ with 
  $V \sub (S_W(V))_{|P}$. 
(Here the sum is over all simple $L_P$-subrepresentations $W$ of $H^0(\fru_P,M)'$ with multiplicities.)
  \end{prop}

\begin{proof}
By Lemma \ref{duality} it is enough to prove the result for the $\fru_P$-invariants of 
the (topological) dual $\cF^G_P(M,V)'.$

\smallskip
For simple objects $M$ we use Theorem  \ref{thm_H0}.
If here the considered parabolic subgroups $P$ and $Q$  are identical then the claim is trivial since
$$H^0(\fru_P, \cF_P^G(M,V)') =  W_M \otimes_K V'.$$
Otherwise, we get
$$H^0(\fru_P, \cF_P^G(M,V)') =  H^0(\fru_P, W_M) \otimes_K i^Q_P(V)'$$
by the $PQ$-formula.
Then we apply \cite[II, Prop. 2.11]{Ja}. The latter reference says that for an 
algebraic  simple $G$-module $M$ the fix space $M^{U_P}$ is a simple $L_P$-module, as well.
Hence the module $L_P$-module $H^0(\fru_P,W_M)=W_M^{U_P}$ is simple and contributes to the index family of the direct sum.

In general we fix a Jordan-H\''older-series of $M$ and apply induction to the number of irreducible subquotients of $M$. More precisely, we
consider an exact sequence 
$$0 \to M_1 \to M \stackrel{p}{\to} M_2 \to 0$$
in our category $\cO^\frp_\alg$ where we suppose that
$M_2$ is simple. Hence we get by applying our bi-functor $\cF^G_P$ composed with taking the dual an exact sequence
$$0 \to \cF^G_P(M_1,V)' \to \cF^G_P(M,V)' \to \cF^G_P(M_2,V)' \to 0.$$ Next we take $\fru_P$-invariants which gives together with the 
induction hypothesis and the start of induction a (left) exact sequence
$$0 \to \bigoplus_{W\sub H^0(\fru_P,M_1)} W \otimes S_W(V)' \to H^0(\fru_P,\cF^G_P(M,V)') \stackrel{g}{\to} W_2 \otimes i^Q_P(V)'$$
 where  $Q$ is maximal for $M_2$ and $W_2=H^0(\fru_P,M_2).$
Next we consider the natural map $$\bar{p}=H^0(\fru_P,p):H^0(\fru_P,M) \to H^0(\fru_P,M_2)$$ and distinguish the following two cases:

\smallskip
\noindent $1^{\rm st}$ case: $\bar{p}=0.$  In this case we claim that the map $g$ also vanishes. Indeed, for seeing this we have to reenter the proof
of Theorem \ref{thm_H0}. Here\footnote{which holds for arbitrary modules $M$.} we saw the identities
$$ H^0(\fru_P,\cF^G_P(M)')= \bigoplus_{w \in W^I}H^0(\mathrm{Ad}(w^{-1}) (\fru_P), D(w^{-1}\cI w) \otimes_{U(\frg,
      w^{-1}Iw \cap P_0)} \otimes M ) $$
and
$$H^0(\mathrm{Ad}(w^{-1}) \fru_P, D(w^{-1}\cI w) \otimes_{U(\frg,w^{-1}Iw \cap P_0)} \otimes 
M )=  \varprojlim_r \sum_u \delta_u \otimes H^0(\mathrm{Ad}((wu)^{-1}) (\fru_p), M)$$
where $u$ varies in $U^-_{P,0}$ depending on $r$.
The map    $H^0(\fru_P,\cF^G_P(M)') \to H^0(\fru_P,\cF^G_P(M_2)')$ is clearly compatible (graded) with respect to the operations
$\bigoplus_{w\in W^I}$, $\sum_r$ and $\varprojlim_r.$

 
 \smallskip
 If $w \not\in Q$ or $u\not\in Q$, then $H^0(\mathrm{Ad}((wu)^{-1})(\fru_P), D(w^{-1}\cI w) \otimes_{U(\frg,
      w^{-1}Iw \cap P_0)} \otimes M_2)=0$ by the proof of Theorem \ref{H0_alg}. In particular, if $P=Q$ we are done.
      
      \smallskip
    If $w \in Q$ and $u\in Q$, then
     $\fru_Q = \mathrm{Ad}((wu)^{-1})(\fru_Q) \subset  \mathrm{Ad}((wu)^{-1}) (\fru_P).$ 
     By assumption the contribution $\delta_u \otimes H^0(\mathrm{Ad}((wu)^{-1}) (\fru_p), M)$ vanishes under $g$ for $w=1,u=1$.
        Suppose on the other hand that for $w\neq 1, u\neq 1,$ there is a non-trivial vector $v\in H^0(\mathrm{Ad}((wu)^{-1}) (\fru_P),M), v\neq 0.$  As $\mathrm{Ad}((wu)^{-1}) (\fru_P)\cap \fru^-_p\neq 0$ there is an element $z\in \fru_P^-$ which  kills $v.$ In particular, $z$ acts locally finitely on
     $v.$ 
     The subspace $N\subset M$ on which $z$ acts locally finitely
     is a $U(\frg)$-submodule, cf. \cite[Lemma 8.2]{OS2} and contains $v$. Thus  $v\in H^0(\mathrm{Ad}((wu)^{-1}) \fru_P,N) \subset
     H^0(\mathrm{Ad}((wu)^{-1}) \fru_P,M).$ 
     Further $N \subsetneq M$ is a proper submodule. Indeed, suppose that $N=M.$ We may assume that $P$ is maximal for $M$. Thus there must be a simple subquotient of $M$ for which $P$ is maximal, cf. \cite[9.3 Prop.]{H1} and on which $z$ acts locally finitely, as well.  But this is not possible by \cite[Cor. 8.7]{OS2}. 
     
     \smallskip
     If $N$ is simple then there are the following 2 cases:
     
     \smallskip
     $\alpha)$ Let $N\subset M_1$. Then clearly $g(v)=0$.
     
     \smallskip
     $\beta)$ Let $N\nsubseteq M_1$. Then $N$ is mapped isomorphically onto $M_2$ giving rise to a splitting of 
     the surjection $M \to M_2.$ A contradiction to the vanishing of the map 
     $\bar{p}.$
     
     \smallskip
     In general we argue on induction on the length on $M$ to see that $g(v)=0$ (Indeed we apply the induction hypothesis to $N$).   
     It follows      that the map $H^0(\mathrm{Ad}((wu)^{-1})(\fru_P),M) \to H^0(\mathrm{Ad}((wu)^{-1})(\fru_P),M_2)$ vanishes, as well.
 Thus $g=0.$   
 
    \medskip      
\noindent $2^{\rm nd}$ case: $\bar{p}\neq 0$. In this case $\bar{p}$ is automatically surjective. Hence we see that there is an exact
sequence 
$$0 \to \bigoplus_{W\sub H^0(\fru_P,M_1)} W \otimes S_W(V)' \to H^0(\fru_P,\cF^G_P(M,V)') \stackrel{g}{\to} W_2 \otimes S_{W_2}(V)' \to 0$$
for some quotient  $S_{W_2}(V)$ of $i^Q_P(V)$. Since $H^0(\fru_P,M)$ is obviously always contained in 
$H^0(\fru_P,\cF^G_P(M)')$  we see that $V \sub (S_{W_2}(V))_{|P}.$ 
On the other hand, it follows from the definition of the category $\cO^\frp$ that  
we have $H^0({\fru_P},M)=H^0({\fru_P},M_1) \bigoplus H^0({\fru_P},M_2)$ as $L_P$-modules with $H^0({\fru_P},M_2)=W_2.$ Since the action of $U_P$ is trivial on these spaces, this identity holds even as $P$-representations. 
It follows that $H^0(\fru_P,\cF^G_P(M,V)') \subset 
H^0(\fru_P,D(G)\otimes_{D(\frg,P)} H^0(\fru_P,M) \hat{\otimes} V').$
(For simple objects this is trivial and in general use induction together with the argument above). Since
$D(G)\otimes_{D(\frg,P)} H^0(\fru_P,M) \hat{\otimes} V'=(D(G)\otimes_{D(\frg,P)} H^0(\fru_P,M_1) \hat{\otimes} V') \bigoplus (D(G)\otimes_{D(\frg,P)} H^0(\fru_P,M_2) \hat{\otimes} V')$
we see  that above sequence splits.

. 
%
\end{proof}

\vskip8pt
We can generalize the previous result as follows. Fix an integer $k\geq 1.$
Let $\fru_P^k \sub U(\frg)$ be the subspace generated by all the products $x_1\cdots x_k$ with $x_i \in \fru_P.$
With the same proof one checks:

\vskip8pt
\begin{prop}\label{H0_uk}
  Let $M\in \cO^\frp_\alg$ and let $V$ be a smooth admissible $L_P$-representation. Then there are canonical (topological)
  identities
   $$H^0(\fru_P^k,\cF^G_P(M,V)')=\bigoplus_{W\sub H^0(\fru_P^k,M)} W \otimes S_W(V)'$$ 
   and
  $$\overline{H}_0(\fru_P^k,\cF^G_P(M,V))=\bigoplus_{W\sub H^0(\fru_P^k,M)} W' \otimes S_W(V)$$ 
  of $P$-representations where 
  $S_W(V)$ is a quotient of $i^{P_W}_P(V)_{|P}$ for some standard parabolic subgroup $P_W\supset P$
  with $V \sub (S_W)(V)_{|P}.$ 
(Here the sum is over all maximal indecomposable  $P$-subrepresentations  $W$ of the finite-dimensional $P$-representation 
$H^0(\fru_P^k,M)$ with multiplicities.)
\end{prop}

\begin{proof}
We start with the remark that Lemma \ref{duality} generalizes to these $\fru_P^k$-invariants so that  $\overline{H}_0(\fru_P^k,\cF^G_P(M,V))$ is the topological dual of $H^0(\fru_P^k,\cF^G_P(M,V)')$ as $P$-representation.
 As already mentioned the proof coincides with that of Proposition \ref{thm_H0_general}. We list here the corresponding modifications.
 
 For  the start of induction which is essentially Theorem \ref{thm_H0} one has to pay attention. Here we follow the proof
 of loc.cit. where $k=1.$ If $w\neq 1$, then some elements of $\fru_P^k$ act injectively on $M_r^w$, too. Hence all the contributions
 $H^0(\mathrm{Ad}(w^{-1}) (\fru_P), D(w^{-1}\cI w) \otimes_{U(\frg,
      w^{-1}Iw \cap P_0)} \otimes M)$
 vanish.
 As for $w=1$ we observe that $H^0(\mathrm{ad}(u^{-1})\fru_P^k, M^1_r) \neq 0$ implies that
  $H^0(\mathrm{ad}(u^{-1})\fru_P, M^1_r)\neq 0.$ Hence we obtain for a simple 
  object $M$ for which $P$ is maximal the identity
 $$H^0(\fru_P^k,\cF^G_P(M,V)')=H^0(\fru_P^k,M) \otimes V'.$$
 The object $H^0(\fru_P^k,M) $ is an indecomposable $P$-module which gives the claim in this case.  
  If $Q\supset P$ is maximal for $M$ we get
  $$H^0(\fru_P^k,\cF^G_P(M,V)')=H^0(\fru_P^k,H^0(\fru_Q^k,M)) \otimes i^Q_P(V)'.$$ But the first factor is again indecomposable since even
  $H^0(\fru_B^k,M)$ is indecomposable as it coincides with the sum 
  $\sum_{i_1+ \cdots + i_d < k} K\cdot y_{\alpha_1}^{i_1}\cdots y_{\alpha_d}^{i_d} \cdot v^+ $,  where  $v^+$ is a highest weight vector generating $H^0(\fru_P,M)$,  $\{\alpha_1,\ldots,\alpha_d\}$ is a root basis and  $y_{\alpha_1}, \ldots, y_{\alpha_d} \in \fru_P^-$ are the usual generators of the weight spaces.).

As for the induction step we note that $\bar{p}: H^0(\fru_p,M) \to H^0(\fru_p,M_2)$ is surjective iff $H^0(\fru_p^k,M) \to H^0(\fru_p^k,M_2)$ is
surjective for all $k$ (by considering the epimorphisms $U(\fru_p^-) \otimes H^0(\fru_p,M) \to M$ and $U(\fru_p^-) \otimes H^0(\fru_p,M_2) \to M_2$, respectively.).
\end{proof}

\begin{rmk}
 Consdering the proof of the above propositions we deduce the following fact for two 
 integers $l >k \geq 1.$ If $W_k \subset H^0(\fru_P^k,M)$ and 
 $W_l \subset H^0(\fru_P^l,M)$ are two maximal indecomposable  $P$-subrepresentations  such that $W_k \subset W_l$ then $S_{W_k}(V)=S_{W_l}(V).$
\end{rmk}

\vspace{0.5cm}

For a locally analytic $T$-representation $V$ and a locally analytic character $\lambda:T \to K^\ast$
we denote by
$$V_\lambda:=\{v\in V \mid tv = \lambda(t)v \, \forall t\in T\}$$
the $\lambda$-eigenspace of $V$. We set
$$V_\alg:=\bigoplus_{\lambda \in X^\ast(T)} V_\lambda.$$

\begin{cor}\label{H0_alg}
 Let $M\in \cO^\frp_\alg$ and $k \geq 1.$ Then $\overline{H}_0(U_P,\cF^G_P(M))_\alg=H^0(\fru_P,M)'$ and 
 $\overline{H}_0(\fru_P^k,\cF^G_P(M))_\alg=H^0(\fru_P^k,M)'$.
\end{cor}

\begin{proof}
Since the weight spaces of  $M$ are algebraic  we see that $(W \otimes S_W({\bf 1}))_\alg=W$ for all contributions $W$
in $H^0(\fru_P^k,M)$. Hence the claim follows.
\end{proof}

In the case of generalized Verma modules we can give a more precise statement.

\begin{prop}\label{h0_verma}
 Let $M=U(\frg)\otimes_{U(\frp)} W\in \cO^\frp_\alg$ be a generalized Verma module for some parabolic subgroup $P$
 and let $V$ be a smooth admissible $L_P$-representation. 
Then $\overline{H}_0(U_P,\cF^G_P(M,V))=H^0(\fru_P,M)'\otimes V$ and 
 $\overline{H}_0(\fru_P^k,\cF^G_P(M,V))=H^0(\fru_P^k,M)'\otimes V$ for all $k\geq 1.$
\end{prop}

\begin{proof}
We may suppose that $V$ is trivial.
 The start of the proof is the same as in Theorem \ref{thm_H0}. For $w\neq 1$ one checks that the contributions
 $H^0(\fru_P^k, D( \cI ) \otimes_{U(\frg, I \cap wP_0w^{-1})} M^w)$
 vanishes well  since a generalized Verma module is free over $U(\fru_P^-)$
 and consequently elements of $\fru_P^-$ act injectively on $M.$
  
    Now consider the case $w=1$. Here we shall show
  that if $u \in U_{P,0}^-\setminus \{1\}$, then we have
  $H^0(\mathrm{Ad}(u^{-1}) (\fru_P), M)=0$. 
  Indeed, let $u\neq 1$. Since the normalisator of $\fru_P$ under the adjoint action of $G$ is the
  parabolic subgroup $P$, there  is some $v \in \fru_P$ such that $uvu^{-1} \not\in \fru_P.$
	Write $uvu^{-1} =v_- +v_+$ where $v_- \in \fru_P^-$ and $v_ +\in \frp.$ Let $m \in M_\chi, m\neq 0$ for some weight $\chi.$
	As we have already used above the action of $\fru_P^-$ is injective on $M$. Hence $v_- \cdot m\neq 0.$ 
	But the elements $v_-, v_+$ shift the weights of $M$  in opposite directions. Any identity $(v_- + v_+)\cdot m=0$ would 
	imply $0\neq v_-m = -v_+m$ which yields thus for weight reasons a
	contradiction. 	In general we decompose any element $m \in M$ into its weight components. For simplicity
	let $m=m_1+m_2$ where $m_i \in M_{\chi_i}$ and $\chi_1\neq \chi_2$ are weights. Again we consider the sequence
	$0\neq v_-m = v_-m_1 + v_- m_2=-v_+m_1 - v_+m_2.$ Comparing weights and that the action of
	$\fru_P^-$ on $M$ is injective we see that this is not possible. Hence $H^0(\mathrm{Ad}(u^{-1})(\fru_P), M)=0$ since $uvu^{-1}$ acts injectively on $M.$ With the same proof as in Step 1 of \cite[Thm. 5.7]{OS3} one checks that $v_-$
	acts injectively on $M_r.$ As the weights for the action of $v_+$ on $M_r$ are different from those of $v_-$, we see that $uvu^{-1}$ acts injectively 
	on $M_r$ as well. Thus 
		$H^0(\mathrm{Ad}(u^{-1})(\fru_P), M_r)=0$.
		By repeating the arguments in Theorem \ref{thm_H0} we obtain an isomorphism of Fr\'echet spaces
  \begin{equation*}H^0(\fru_P, D( \cI ) \otimes_{U(\frg,P_0)} M )
    \simeq H^0(\fru_P, M).
  \end{equation*}
  
  The claim follows moreover for all $k\geq 1$ inductively since $\Ad(u^{-1})(\fru_P^k)=\Ad(u^{-1})(\fru_P)^k$. 
 \end{proof}

 \begin{exam}
  Let $G=\GL_2$ and consider the exact sequence $0 \to M(s\cdot0) \to M(0) \to {\bf 1} \to 0.$ By applying 
  of $\cF^G_B$ and dualizing   we get an exact sequence
  $$0\to D(G) \otimes_{D(\frg,B)} M(s\cdot0) \to D(G) \otimes_{D(\frg,B)} M(0) \to (i^G_B({\bf 1}))' \to 0. $$
  Taking $H^0(\fru_B,-)$-invariants we get a left exact sequence
  $$0 \to K_{s\cdot 0} \to K_{s\cdot 0} \bigoplus K_{0} \to (i^G_B({\bf 1} ))'$$
  which is not exact.
 \end{exam}

%

\begin{rmk}\label{geht_auch}
 The same statement holds true (with the same proof) for objects $M\in \cO^\frp_\alg$ of the shape 
 $M=U(\frg) \otimes_{U(\frp) }W$ where
 $W$ is an arbitrary finite-dimensional algebraic $P$-representation. 
 In particular, it holds for objects $M$ which are  projective in the category $\cO_\alg^\frp$ since
 such an object it is free as a $U(\fru_P^-)$-module \cite{H1}.  
\end{rmk}

\vskip8pt

Next there is the following variant of the above proposition concerning the other parabolic  subgroups of type $P$ 
lying in the same apartment induced by $T$.
Let $P=P_I=L_PU_P$ and set for $w\in W^I$, $P^w=w^{-1}Pw ,L_P^w=w^{-1}L_Pw, U_P^w=w^{-1}U_Pw.$ Here for a $L_P^w$-module 
$V$, we let $V^{w}$ be the $L_P$-module twisted by $w,$ i.e., we consider the action
induced by composing the given action with the homomorphism $L_P \to w^{-1}L_Pw, g\mapsto w^{-1}gw$.

\begin{prop}\label{h0_verma_alle}
With the above notation, let $M\in \cO^{\frp^w}_\alg$ be a generalized Verma module  with respect to $P^w$ or a simple module 
such that $P^w$ is maximal for $M$. Let $V$ be a smooth admissible $L_P^w$-representation. 
Then $\overline{H}_0(U_P,\cF^G_{P^w}(M,V))=(H^0(\fru_P^w,M)')^{w}\otimes V^{w}$ and
$\overline{H}_0(\fru_P^k,\cF^G_{P^w}(M,V))=(H^0((\fru_P^w)^k,M)')^{w}\otimes V^{w}$ for all $k\geq 1.$
\end{prop}

\begin{proof}
 The proof is the same as above. The difference is that this time all contributions 
 
 $H^0(\mathrm{Ad}(x^{-1}) \fru_P^k, D(x^{-1}\cI x) \otimes_{U(\frg,
      x^{-1}\cI x \cap P_0^w)} \otimes M )$
 with $x\neq w$ vanish. 
\end{proof}

  Next we consider an analogue of the Casselman-Jacquet functor \cite{Ca2}, i.e., limits of the above  functors $H^0(\fru_P^k,-)$ (resp.  $\overline{H_0}(\fru_P^k,-)$ by duality) with 
  varying $k$.   For a locally analytic $G$-representation $U$, the expression $\varinjlim_k H^0(\fru_P^k,U')$
  is a $\frg \rtimes P$-module as the  same reasoning as in loc.cit. applies.
    We denote by $$\cG^G_P: \Rep_K(G)^{\rm loc.an.} \to {\rm Mod}_{\frg\rtimes P}$$
    the induced functor.
  As before let $M$ be an object of $\cO^\frp_\alg$ and let $V$ be a smooth admissible $L_P$-representation. Then the object 
  $\varinjlim_k H^0(\fru_P^k,\cF^G_P(M,V)')$ is even a $D(\frg, P)$-module 
  since $M$ is an inductive limit of finite-dimensional $P$-representations.
    In this way we get in some sense a right adjoint to the globalisation functor $\cF^G_P$. Moreover,
   it defines a "section" of it for some objects in $\cO^\frp_\alg \times \Rep^{\infty,a}_K(L_P)$ (i.e. $\cG^G_P(\cF^G_P(M,V))=M \otimes V')$ , cf. 
   Proposition \ref{h0_verma} and Theorem \ref{thm_H0}.
      In general we can deduce the following statement.
   
   \begin{prop}\label{G}
  Let $M\in \cO^\frp_\alg$ and let $V$ be a smooth admissible $L_P$-representation. Then there is a canonical (topological)
  identity
   $$\cG^G_P(\cF^G_P(M,V))=\bigoplus_{N \sub M} N \otimes S_N(V)'$$ 
    of $P$-representations where 
  $S_N(V)$ is a quotient of $i^{P_N}_P(V)_{|P}$ with $V \sub (S_N)(V)_{|P}$ for the uniquely standard parabolic subgroup $P_N\supset P$ which is maximal for $N.$ (Here the sum is over all simple constituents $N$ of $M$ with multiplicities.)
\end{prop}
   
   \begin{proof}
    This follows from Proposition \ref{H0_uk} by taking the inductive limit. Note that $\varinjlim_k H^0(\fru_p^k,M)=M.$
   \end{proof}

   \vskip8pt
   
     \begin{prop}\label{U_irr_subquotient}
      Let $U$ be some irreducible subquotient of some $\cF^G_P(M,V)$ with $M \in \cO^\frp_\alg.$  
      Then $\cG^G_P(U)$ is a simple $D(\frg,Q)$-module for some parabolic subgroup $Q\subset G$ with $P\subset Q$.
   \end{prop}
   
   \begin{proof}
   Since $U$ is simple it must coincide by the JH-theorem applied to $\cF^G_P(M,V)$
   with some object of the shape $\cF^G_Q(N,W)$ where $N$ is a simple subquotient of $M$, $Q$ is maximal for
   $N$ and $W$ is an irreducible subquotient of $i^Q_P(V).$
   But for these objects we deduce by Proposition \ref{G} that $\cG^G_Q(U)=N\otimes W'$.  On the other hand, we
   have $\cG^G_Q(U)=\cG^G_P(U)$ since any element of $N$ is killed by weight reasons by some $\fru_P^k$, $k \geq 1.$
  Hence we get the claim.
   \end{proof}

   As a by-product we get the following statement by applying the functor $\cG^G_P$ and
   Proposition \ref{H0_uk}. One part of it was already given by Breuil \cite[Cor. 2.5]{Br}.
   
   \begin{cor}\label{U_simple}
    Let $U$ be an irreducible subobject (quotient) of some $\cF^G_P(M,V)$ with $M \in \cO^\frp_\alg.$
    Then $U$ has the shape $\cF^G_Q(N,W)$ where $P \sub Q$ and $N$ is a simple quotient of $M$ (submodule) and 
    $W$ is an irreducible subrepresentation (quotient) of $i^Q_P(V)$. 
   \end{cor}
   
   \begin{proof} In the proof of the foregoing proposition we saw that $U$ has the shape $\cF^G_Q(N,W)$ where $N$ is a simple subquotient of $M$, $Q$ is maximal for
   $N$ and $W$ is an irreducible subquotient of $i^Q_P(V).$
    If $U$ is a quotient then we get by the left exactness of the functor $\cG^G_P$ an injection
    $\cG^G_P(U) \hookrightarrow \cG^G_P(\cF^G_P(M,V))$ of $D(\frg,P)$-modules.
    In particular, we see that $N$ is a submodule of $M$
    since $\cG^G_P(U)=N\otimes W'$ . Further we see, e.g., by Proposition \ref{G} that  $W'\subset S_N(V)=i^Q_P(V)$. The claim follows.
    
    If $i:U \hookrightarrow \cF^G_P(M,V)$ is a subobject we get a morphism  $\cG^G_P(i):\cG^G_P(\cF^G_P(M,V)) \to \cG^G_P(U).$ 
   
    The dual of $i$ is a homomorphism $i':D(G)\otimes_{D(\frg,P)} M \hat{\otimes} V' \to D(G)\otimes_{D(\frg,Q)} N \hat{\otimes} W'$ 
    which gives rise by the very definition (taking $\varinjlim_k$ of $\fru_p^k$-invariants)
    to the morphism $\cG^G_P(i).$ Since $M\otimes V'=\varinjlim_k H^0(\fru_P^k,M)\otimes V'$ we see that the composite of the natural map
    $M\otimes V' \to \cG^G_P(\cF^G_P(M,V))$ with $\cG^G_P(i)$ is a map $M\otimes V' \to N\otimes W'.$ The latter one  
    induces by taking the composite of $D(G) \otimes_{ D(\frg,P)} -$ with the natural map  $D(G)\otimes_{D(\frg,P)} N \otimes W' \to D(G)\otimes_{D(\frg,Q)} N \otimes W'$  the map $i'$. In particular the map 
    $M\otimes V' \to N\otimes W'$ is non-trivial. By Frobenius reciprocity we get a non-trivial homomorphism 
    $M\otimes i^Q_P(V)' \to N\otimes W'$ of $\frg \times Q$-modules. As the latter module is simple we obtain surjections $M \to N$ and
    $i^Q_P(V)' \to W'.$ The claim follows.
   \end{proof}

\vspace{0.5cm}

\section{Are the functors \texorpdfstring{$\cF^G_P$}{} faithful?}\setcounter{subsection}{1}

In this section we want to address the question whether the functors $\cF^G_P$ are faithful resp. fully
faithful. This aspect was discussed for $G={\rm SL}_2$ already in the series of papers by Morita \cite{Mo1,Mo2,Mo3}.

\begin{thm}\label{fullyfaithful}
 Let $M_1, M_2\in \cO^\frp_\alg$. 
 Then the map
 \begin{eqnarray*} \Hom_{\cO^\frp_\alg}(M_1,M_2) & \to &  \Hom_G(\cF^G_P(M_2),\cF^G_P(M_1)) \\
 f & \mapsto &  \cF^G_P(f) 
 \end{eqnarray*}
 is bijective. 
\end{thm}


\begin{proof}
The proof is divided into several steps.
\vskip8pt
1) Let $M_1=M(Z)$ be a generalized  Verma module for some finite-dimensional algebraic
 $L_P$-representation $Z$. Then $\cF^G_P(M_1)=\Ind^G_P(Z')$ and $U_P$ acts trivially on $Z.$  
Now we have $\overline{H}_0(U_P,\cF^G_P(M_2))_\alg= H^0(\fru_P,M_2)'$ by Lemma \ref{H0_alg}.
We consider the  identities induced by Frobenius reciprocity and the previous observations 
\begin{eqnarray*}
 \Hom_G(\cF^G_P(M_2),\cF^G_P(M_1)) & = & \Hom_P(\cF^G_P(M_2),Z') \\  
 & = & \Hom_{L_P}(\overline{H}_0(U_P,\cF^G_P(M_2)),Z')\\
& = & \Hom_{L_P}(\overline{H}_0(U_P,\cF^G_P(M_2))_\alg,Z')\\
 & \cong & \Hom_{L_P}(H^0(\fru_P,M_2)',Z')\\
  & \cong & \Hom_{L_P}(Z, H^0(\fru_P,M_2)) \\
  &= & \Hom_{P}(Z,M_2) \\
  & = & \Hom_{\cO^\frp_\alg}(M_1,M_2). 
 \end{eqnarray*}
Here the third equality follows from that fact that $Z$ is algebraic and $\Hom_G(i^G_P,{\bf 1})=0$ for all parabolic
subgroups $P \subsetneq G.$ .
 
\vskip8pt
2) Let $M_1$ be a quotient of some generalized Verma module, i.e., there is a surjective homomorphism 
$M(Z) \to M_1$ for some finite-dimensional algebraic
 $L$-representation $Z$.  Let  $\frd$ be its kernel. Then by definition we have $\cF^G_P(M_1)=\cF^G_P(M(Z))^\frd$. 
We consider the commutative diagram

\begin{eqnarray*}
 \Hom_{D(\frg,P)}(M_1,M_2) & \hookrightarrow & \Hom_{D(\frg,P)}(M(Z),M_2) \\
 \downarrow & & \downarrow \\
  \Hom_G(\cF^G_P(M_2),\cF^G_P(M_1)) & \hookrightarrow  &  \Hom_G(\cF^G_P(M_2),\cF^G_P(M(Z))).
\end{eqnarray*}

\noindent By Step 1) the right vertical  map is an isomorphism. It follows that the left vertical map is injective.
To show surjectivity we consider the dual objects, i.e. the commutative diagram 

\begin{eqnarray*}
 \Hom_{D(\frg,P)}(M_1,M_2) & \hookrightarrow & \Hom_{D(\frg,P)}(M(Z),M_2) \\
 \downarrow & & \downarrow \\
  \Hom_{D(G)}(M_1^{D(G)}, M_2^{D(G)}) & \hookrightarrow  &  
  \Hom_{D(G)}(M(Z)^{D(G)}, M_2^{D(G)}).
\end{eqnarray*}

\vskip8pt
\noindent where we abbreviate $M^{D(G)}:=D(G)\otimes_{D(\frg,P)} M$ for $M \in \cO^\frp_\alg.$
Moreover the vertical maps are the obvious ones, i.e. induced by base change. For the surjectivity, let 
$f\in  \Hom_{D(G)}(M_1^{D(G)}, M_2^{D(G)})$ and consider it via the injection 
as an element
in the set $\Hom_{D(G)}(M(Z)^{D(G)}, M_2^{D(G)}).$ Hence there is some
morphism $\check{f}:M(Z) \to M_2$ with $\check{f} \otimes \id=f.$ We need to show
that $\check{f}(\frd)=0.$ By assumption we have $f(\frd)=0.$ 
But we proved in \cite[(3.7.6)]{OS3} that
if $M\in \cO_\alg, M\neq 0$ then $D(G) \otimes_{D(\frg,B)} M \neq 0.$ By applying this fact to $M=\check{f}(\frd)$  the claim
follows.

\vskip8pt
3) Let $M_1=U(\frg)\otimes_{U(\frp)} W$ for some finite 
dimensional algebraic $P$-representation $W$. We may view it as a successive extension of generalized Verma modules
considered in Step 1). 
The proof of the statement is by dimension on $\dim W.$ Here Step 1) serves as the start of induction.
Write down an exact sequence
$$0\to M(Z) \to M_1 \to M_1'  \to 0 $$
where $M_1'=U(\frg)\otimes_{U(\frp)} W'$ is a generalized Verma module for some algebraic $P$-representation $W'$ with  
$\dim W'<\dim W.$ We get an induced exact sequence
$$0\to \cF^G_P(M_1')\to \cF^G_P(M_1) \to \cF^G_P(M(Z))  \to 0 .$$
We consider the resulting diagram of long exact sequences 

\vskip8pt

{\scriptsize

$$\begin{array}{cccccc}
  0 \to & \Hom_{D(\frg,P)}(M_1',M_2) & \to & \Hom_{D(\frg,P)}(M_1,M_2) & \to 
 & \Hom_{D(\frg,P)}(M(Z),M_2) \\
 & \downarrow f' &  & \downarrow f & & \downarrow f_Z  \\
 0 \to  & \Hom_G(\cF^G_P(M_2),\cF^G_P(M_1')) & \to 
 &  \Hom_G(\cF^G_P(M_2),\cF^G_P(M_1))
  & \to &  \Hom_G(\cF^G_P(M_2),\cF^G_P(M(Z)))  \\ \\
    \stackrel{\delta}{\to}  & {\rm Ext}^1(M_1',M_2) & \to & {\rm Ext}^1(M_1,M_2) & \to 
 & {\rm Ext}^1(M(Z),M_2) \\
 & \downarrow & & \downarrow & & \downarrow  \\
  \stackrel{\delta_{\cF}}{\to} & {\rm Ext}^1(\cF^G_P(M_2),\cF^G_P(M_1')) & \to 
 &  {\rm Ext}^1(\cF^G_P(M_2),\cF^G_P(M_1))
  & \to &  {\rm Ext}^1(\cF^G_P(M_2),\cF^G_P(M(Z))) 
\end{array}$$}

\normalsize

\vskip8pt
\noindent Here we consider the Ext groups as Yoneda-Ext groups. The maps $f'$ and $f_Z$ are by induction isomorphisms of finite-dimensional vector spaces. 
By diagram chase, it suffices to check that $\delta(g)\neq 0$ if and only if $\delta_{\cF}(\cF^G_P(g)) \neq 0.$
Concretely we have to show that if $\delta(g) \neq 0$ then $\delta_{\cF}(\cF^G_P(g)) \neq 0$ since the other direction 
follows directly by diagram chase again. If $\delta_{\cF}(\cF^G_P(g))=0$, then the extension 
$$0\to \cF^G_P(M'_1) \to E_{\cF^G_P(g)} \to \cF^G_P(M_2) \to 0  $$
induced by $\cF^G_P(g)\in \Hom_G(\cF^G_P(M_2),\cF^G_P(M(Z))$ 
splits. Then we apply Remark \ref{geht_auch} to deduce that 
\begin{eqnarray*}
H^0(\fru,E_{\cF^G_P(g)}) & = & H^0(\fru,\cF^G_P(M'_1)) \oplus H^0(\fru,\cF^G_P(M_2))
\\ & = & H^0(\fru,M'_1) \oplus H^0(\fru,M_2).
\end{eqnarray*}
Since $H^0(\fru,E_g)\subset H^0(\fru,E_{\cF^G_P(g)})$ we deduce that $H^0(\fru,E_g)=  H^0(\fru,E_{\cF^G_P(g)})$
and  that the extension
$$0 \to M_2 \to E_g \to M_1' \to 0 $$
splits as well. Indeed suppose for simplicity that $W'$ is induced via inflation by a $L_P$-representation. Then $W'=H^0(\fru,M'_1)$
appears in $E_g$ so that we get a  section of $E_g \to M_1'.$

\vskip8pt
4) Let $M_1$ be arbitrary. Then there is a surjective homomorphism $M(Z) \to M$ for some finite 
dimensional algebraic $P$-representation $Z$. Then we proceed as in Step 2).

\end{proof}

Next we consider the situation where also smooth admissible representations are involved.

\begin{prop}\label{remark_conj_2}
Let $M_1,M_2\in \cO^\frp_\alg$ and let $V_1,V_2$ be smooth admissible $L_P$-represen\-tations. Assume that
$Z\sub M_1$ is a finite-dimensional algebraic $P$-representation which generates $M_1$ as a $U(\frg)$-module and $k \geq 1$ an integer such that $H^0(\fru_P^k,Z)=Z.$
Then the natural map
 \begin{eqnarray*} \Hom_{\cO^\frp_\alg}(M_1,M_2) \otimes \Hom_{L_P}(V_2,V_1) & \to &  \Hom_G(\cF^G_P(M_2,V_2),\cF^G_P(M_1,V_1)) 
 \end{eqnarray*}
 induced by the functor $\cF^G_P$
 is injective and extends to a bijection
$$ \bigoplus_{W \subset H^0(\fru_P^k,M_2)} \Hom_{\cO^\frp_\alg}(M_1,M_2)_W \otimes \Hom_P(S_W(V_2)_{|P},V_1)  \to  \Hom_G(\cF^G_P(M_2,V_2),\cF^G_P(M_1,V_1))$$
where $W$ ranges over all maximal indecomposable $P$-modules of $H^0(\fru_P^k,M_2)$ such that there is a non-trivial homomorphism $Z \to W$
and $\Hom_{\cO^\frp_\alg}(M_1,M_2)_W \subset 
\Hom_{\cO^\frp_\alg}(M_1,M_2)$ is just the subspace consisting of those maps which are induced by this homomorphism.
\end{prop}

\begin{proof}
Indeed we consider Steps 1) and  3) in the modified situation. Then we argue as in Steps 2) and  4) for the general case.
As for Steps 1) and 3) we choose this time a slightly different approach by way of variation.
So, let $M_1=U(\frg)\otimes_{U(\frp)} Z$ for some finite-dimensional $P$-module $Z$. 
Let $k\geq 1$ be an integer such that $H^0(\fru_P^k,Z)=Z.$ Then we apply Proposition \ref{H0_uk} to deduce that
\begin{eqnarray*}
 \Hom_G(\cF^G_P(M_2,V_2),\cF^G_P(M_1,V_1)) & = & \Hom_P(\cF^G_P(M_2,V_2),Z'\otimes V_1) \\  
 & = & \Hom_{D(P)}(Z\otimes V_1',\cF^G_P(M_2,V_2)')\\
 & \cong & \Hom_{D(P)}(Z\otimes V_1',H^0(\fru_P^k,\cF^G_P(M_2,V_2))')\\
  & \cong & \Hom_{D(P)}(Z\otimes V_1', \bigoplus_{W\sub H^0(\fru_P^k,M_2)} W \otimes S_W(V_2)_{|P}')  \\
  & = & \bigoplus_{W\sub H^0(\fru_P^k,M_2)} \Hom_{\cO^\frp_\alg}(M_1,M_2)_W \otimes \Hom_{P}(S_W(V_2)_{|P},V_1) . 
 \end{eqnarray*}
 \end{proof}

\begin{rmk}\label{PQ_geht_auch}
 The statement above is also true (with the same proof) if we consider additionally a parabolic subgroup $Q\supset P$
 such that $M_2\in \cO^\frq_\alg$, $V_2 \in \Rep^{\infty,a}(L_Q)$  i.e. we have a bijection
 $$\bigoplus_{W\sub H^0(\fru_P^k,M_2)} \Hom_{\cO^\frp_\alg}(M_1,M_2)_W \otimes \Hom_P(S_W(V_2)_{|P},V_1)  \to  \Hom_G(\cF^G_Q(M_2,V_2),\cF^G_P(M_1,V_1)).$$
 \end{rmk}



\vskip8pt

 The following example shows that in the general case of objects in $\cO_\alg^\frb$, the first map in 
 Proposition \ref{remark_conj_2} need not to be surjective.

\begin{exam}\label{exam}
Let $G ={\rm SL}_2$, $B\sub G$ the Borel subgroup of upper triangular matrices and let $T = \{\diag(a,a^{-1})\mid a \in L^\ast\}$ 
be the diagonal torus. We consider the smooth character $\chi$ of $T$ given by
$$\chi(\diag(a,a^{-1}) = |a|(-1)^{val_\pi(a)}$$
where $\pi$ is our fixed uniformizer of $O_L$ and $v$
is the normalized valuation, i.e.  $v(\pi) =1$.

Let $M$ be the one-dimensional trivial ${\rm Lie}(G)$-representation. 
Then the object $\cF^G_B(M,\chi)$ is just the smooth representation $i^G_B(\chi).$
But the character $\chi$ is chosen in such a way that $i^G_B(\chi)$ decomposes as a direct sum 
of two irreducible representations \cite[Cor. 9.4.6 (b)]{Cas}.
Hence $\Hom_G(\cF^G_B(M),\cF^G_B(M))$ is two-dimensional whereas $\Hom_{\cO_\alg}(M,M)\otimes \Hom_T(\chi,\chi)$
is one-dimensional.

\end{exam}

 Recall the definitions before Proposition \ref{h0_verma_alle}. For $w\in W$, we denote by $P^w$ the conjugated parabolic subgroup $w^{-1}Pw.$ If $Z$ is
 a finite-dimensional locally analytic representation of $L$ we let $M_w(Z)$ be the corresponding
 generalized Verma module
 with respect to $P^w$, i.e. $M_w(Z)= U(\frg) \otimes_{U(\frp^w)} Z.$
 
 \begin{prop}\label{Hom_twisted_Verma}
   Let  $Z$ be a finite-dimensional algebraic $L_P$-representation and let $w \in W.$ Then
   for any finite-dimensional algebraic $L_P^w$-representation $Y$ there is an identity
 $${\rm Hom}_G\big(\Ind^G_{P^w}(Y'), \Ind^G_P(Z')\big)=
 \Hom_{\cO^{\frp^w}}(M_w(Z^{w^{-1}}),M_w(Y)).$$
 \end{prop}
 
 \begin{proof}
  We argue as in Step 1) in the proof of Theorem \ref{fullyfaithful} and use additionally Proposition \ref{h0_verma_alle}: 
\begin{eqnarray*}
 {\rm Hom}_G\big(\Ind^G_{P^w}(Y'), \Ind^G_P(Z')\big) 
 & = & {\rm Hom}_P\big(\Ind^G_{P^w}(Y'), Z'\big) \\  
 & = & \Hom_L(\overline{H}_0(U_P,\Ind^G_{P^w}(Y')), Z')\\
 & \cong & \Hom_{L}(H_0(\fru_P^w,(M_w(Y)')^w),Z')\\
  & \cong & \Hom_{L}(Z, H^0(\fru_P^w,M_w(Y)^w) \\
  & = & \Hom_{L}(Z^{w^{-1}}, H^0(\fru_P^w,M_w(Y))) \\
  &= & \Hom_{P^w}(Z^{w^{-1}},M_w(Y)) \\
  & = & \Hom_{\cO^{\frp^w}}(M_w(Z^{w^{-1}}),M_w(Y)). 
 \end{eqnarray*}
 \end{proof}

\vspace{0.5cm}

\section{Applications}

In the remaining paper we discuss some applications of the material collected in the previous sections. 

\subsection{The category $\cF^P_\alg$}

We begin to recall a definition of \cite{OS3}.
Let $\lambda, \mu:T \to K^ \ast$ be two algebraic characters with derivatives $d\lambda$, $d\mu$, respectively. 
We write $\mu \uparrow_B \lambda$ if and only if  $d\mu \uparrow_{\frb} d\lambda$ in the sense of \cite{H1}.
Then one has  
\begin{equation}\label{dim1}
 \dim_K {\rm Hom}_{\cO^\frb_\alg}(M(\mu), M(\lambda)) = \left\{\begin{array}{cc} 1 & \mu \uparrow_B \lambda 
 \\ \\0 & otherwise 
\end{array}\right. .
\end{equation}

For the remainder, we denote for an element $w \in W$ and an algebraic character by $w\cdot_B \lambda$ the usual 
``dot''-operation with respect to $B$.
If $\lambda$ is $B$-dominant, then $w\cdot_B \lambda \uparrow_B \lambda$ for all $w \in W.$

On the other hand, we let $\lambda^w:=w(\lambda)$ be the character given by the ordinary action of $W$.

\begin{cor}
 Let $P=B$ and let $\lambda,\mu: T \to K^\ast$ be algebraic characters. Then 
 $$\dim_K {\rm Hom}_G\big(\cF^G_B(M(\lambda)), \cF^G_B(M(\mu))\big)=\left\{\begin{array}{cc} 1 & \mu \uparrow_B \lambda  \\0 & otherwise 
\end{array}\right.$$
 \end{cor}

 \begin{proof}
  This follows from Theorem \ref{fullyfaithful} together with identity (\ref{dim1}). 
 \end{proof}

For a standard parabolic subgroup $P\sub G$, we let $\cF^P_\alg$  be the full subcategory of $\Rep_K^{\rm loc. an.}(G)$ 
consisting of locally analytic representations
which lie in the essential image of the functor  $\cF^G_P:\cO^\frp_\alg \to \Rep_K^{\rm loc.an.}(G)$.

\begin{cor}
i) The category $\cF^P_\alg$ is abelian and has enough injective and projective objects. For a morphism $f: N \to M$ we have
$\cF^G_P({\rm coker}(f))=\ker(\cF^G_P(f))$ and  $\cF^G_P({\rm ker}(f))={\rm coker}(\cF^G_P(f)).$ 

\vskip8pt
ii)   Let $M$ be a projective (resp. injective) object in $\cO^\frp_\alg$. Then $\cF^G_P(M)$ is 
  injective (resp. projective)
  in the category $\cF^P_\alg$.
\end{cor}

\begin{proof}
The category $\cO^\frp_\alg$ is abelian and has enough projective and injective objects. This follows for $\cO^\frp$
 from \cite{H1}. But the proof shows that for an object $M\in \cO^\frp_\alg$ the construction of a projective cover $N$ of
 $M$, that  $N$ is again in the subcategory $\cO^\frp_\alg.$ hence the claim is true
 for the category $\cF^P_\alg$.
Since the functor $\cF^G_P$
induces an equivalence of categories between $\cO^\frp_\alg$ and $\cF^P_\alg$ we get the first part of i)
and ii).
The remaining statements follow directly be the exactness of $\cF^G_P.$
\end{proof}

 We define a dual object for objects lying in the functor. In light of Theorem \ref{fullyfaithful} it is well-defined.

\begin{dfn}
 Let $M\in \cO^\frp_\alg$ and let $M^\vee \in \cO^\frp_\alg$ be its dual object. Set
 $$\cF^G_P(M)^\vee:=\cF^G_P(M^\vee).$$ 
\end{dfn}

\vskip8pt
It follows from the previous corollary that for an object $M\in \cO^\frp_\alg$ the locally analytic $G$-representation 
$\cF^G_P(M)$ is projective (resp. injective) object in $\cF^P_\alg$ if and only if 
$\cF^G_P(M)^\vee$ is injective (resp. projective) object in $\cF^P_\alg$.

  


\begin{dfn}
Let $V_1,V_2 \in \cF^P_\alg$ be two locally analytic representations. We denote by
$\Ext^i_{\cF^P_\alg}(V_1,V_2)$ the corresponding Ext-group in degree $i.$
\end{dfn}

 These Ext-groups are of course different from those considered more generally in the category of locally analytic 
$G$-representations, cf. \cite{K2}. This can be seen as an analogue of relating the groups $\Ext^i_{\frg}(M_1,M_2)$
and $\Ext^i_{\cO}(M_1,M_2)$ for two objects $M_1,M_2 \in \cO$ as the next statement confirms.

\begin{cor} Let $M_1,M_2 \in \cO^\frp_\alg.$
  The natural map
 \begin{eqnarray*} \Ext^i_{\cO^\frp_\alg}(M_1,M_2) & \to &  \Ext^i_{\cF^P_\alg}(\cF^G_P(M_2),\cF^G_P(M_1))
 \end{eqnarray*}
 is bijective.\qed
\end{cor} 

At this point one can derive many consequences on the above defined Ext-groups. Here we exemplary mention only the following:

\begin{cor} Let $\lambda$ be a  dominant algebraic character and let $w,w'\in W.$

\vskip8pt
 a) Unless $w'\cdot \lambda \uparrow w\cdot\lambda$ we have for all $n>0$,
 $$\Ext^n_{\cF^B_\alg}(\cF^G_B(M(w\cdot \lambda)),\cF^G_B(M(w'\cdot \lambda))=0=\Ext^n_{\cF^B_\alg}
 (\cF^G_B(\uL(w\cdot \lambda)),\cF^G_B(M(w'\cdot \lambda)). $$
 
 b) If $w'\cdot \lambda \leq w\cdot \lambda$, then for all $n> \ell(w') -\ell(w)$ 
  $$\Ext^n_{\cF^B_\alg}(\cF^G_B(M(w\cdot \lambda)),\cF^G_B(M(w'\cdot \lambda))=0=
  \Ext^n_{\cF^B_\alg}(\cF^G_B(\uL(w\cdot \lambda)),\cF^G_B(M(w'\cdot \lambda)). $$
\end{cor}

\begin{proof}
 This is a consequence of \cite[Proposition 6.11]{H1}.
\end{proof}

\subsection{The category $\overline{^\infty\cF^P_\alg}$}

\vskip8pt
Next we consider additionally smooth representations as arguments in the functor $\cF^G_P$.
So let $V$ be a smooth $G$-representation. We supply $V$ with the finest locally convex topology.
This approach is compatible with the one of Schneider and Teitelbaum 
for admissible smooth representations. \cite[Section 2]{ST4}.
Equivalently, if we write $V=\bigcup_n V^{G_n}$ for a system of compact
open subgroups $G_n\sub G$ and supply each $V^{G_n}$ with the finest locally convex topology, then the topology on $V$ 
coincides with
the induced locally convex limit topology. 
It  is Hausdorff \cite[Prop. 5.5 ii)]{S1} and barreled 
\cite[Cor. 6.16, Examples iii)]{S1} (see also the construction in \cite[7.1]{Em1}). Moreover, 
for any $v\in V$ the orbit map $G\to V$ is locally constant and gives rise to an element of $C^{an}(G;V)$.
Hence we may and will consider $V$ with the structure of a locally analytic $G$-representation. 
Then $\cF^G_P$ extends with the same definition as in (\ref{display-defF}) to a bi-functor
$$\cF^G_P: \cO^\frp_\alg \times \Rep^\infty_K(L_P) \lra \Rep^{\rm loc. an.}_K(G).$$ 
where  $\Rep^\infty_K(L_P)$ is the category of smooth $L_P$-representations.

\begin{lemma}
 Let $V$ be of countable dimension. Then $\cF^G_P(M,V)$ is of compact type. 
\end{lemma}

\begin{proof}
 We have the following inclusions of closed subspaces
$$\cF^G_P(M,V)\subset \Ind^G_P(W' \otimes V) \cong C^{an}(H,W'\otimes V).$$
Here $H\subset G$ is a locally analytic section of the projection $G \to G/P$, i.e.   $ H \stackrel{\sim}{\longrightarrow} G/P$ and $W$ 
is as usual a finite-dimensional algebraic
$P$-module which generates $M.$
Since a closed subspace of a space of compact type is again of compact type. it suffices to show that this property holds true
for $C^{an}(G,W'\otimes V).$ Now we  may write $V=\varinjlim_n V_n$ as a locally convex limit with finite-dimensional
vector spaces. Hence $C^{an}(H,W'\otimes V)=\varinjlim_n C^{an}(H,W'\otimes V_n).$ But each subspace 
$C^{an}(H,W'\otimes V_n)$ is of compact type and the inductive limit of compact type spaces with injective transition maps
is of compact type again.
\end{proof}

\begin{rmk}
We stress that apart possible from Proposition \ref{U_irr_subquotient} and Corollary \ref{U_simple} (since the proofs do not apply) all results  of the previous sections are also valid for objects lying in the image of this 
enhanced functor. 
\end{rmk}

We denote by $\Rep_K^{\rm \infty,\infty}(G)$ the full subcategory of $\Rep_K^{\rm \infty}(G)$ whose objects are of countable
dimension. This is clearly an abelian subcategory of  $\Rep_K^{\rm \infty}(G)$ which is closed under (smooth) duals.

\begin{lemma}
 The category $\Rep_K^{\rm \infty,\infty}(G)$ has enough injective and projective objects.
 \end{lemma}

\begin{proof}
  Let $V$ be an object of $\Rep_K^{\rm \infty,\infty}(G)$. Since the smooth dual of an injective object
  is projective and vice versa, it suffices to check that $V$ has an embedding into an injective object of
  countable dimension.  For this we consider the injective object $\Ind^{\infty,G}_{\{e\}}(V_{|\{e\}})$ in the larger category
   $\Rep_K^{\rm \infty}(G)$   together with the embedding $V \hookrightarrow \Ind^{\infty,G}_{\{e\}}(V_{|\{e\}}), v \mapsto [g \mapsto gv]$.
  As $G$ is a second countable group and $V$ has a countable basis, we deduce that  
  $\Ind^{\infty,G}_{\{e\}}(V_{|\{e\}})$ is of countable dimension, hence the claim.
\end{proof}

We define $^\infty\cF^P_\alg$ to  be the full subcategory of $\Rep_K^{\rm loc. an.}(G)$ 
consisting of locally analytic representations
which lie in the essential image of the functor
$$\cF^G_P: \cO^\frp_\alg \times \Rep^{\infty,\infty}_K(L_P) \lra \Rep^{\rm loc. an.}_K(G).$$  
The category $^\infty\cF^P_\alg$ is not abelian as we saw for instance in 
Example \ref{exam}.  Concerning the latter aspect,  we consider the smallest abelian subcategory\footnote{which can be seen as
a kind of  Satake compactification of 
$^\infty\cF^P_\alg$} $\overline{^\infty\cF^P_\alg}$
containing all categories $^\infty\cF^Q_\alg$ where $Q\supset P$ is a parabolic subgroup.

   \begin{lemma}\label{intersection}
   Let $M_1,M_2,M \in \cO^\frp_\alg$ and $V_1,V_2,V \in \Rep^\infty_K(L_P)$ such that $M_1,M_2$  are quotients of $M$ and 
   $V_1,V_2$ are subrepresentations of $V$. Then
   $$\cF^G_P(M_1,V_1) \cap \cF^G_P(M_2,V_2)=\cF^G_P(M_1 \oplus_M M_2,V_1\cap V_2)$$
  \end{lemma}

\begin{proof}
 We have $$\cF^G_P(M_1,V_1) \cap \cF^G_P(M_2,V_2)=\cF^G_P(M_1,V_1\cap V_2) \cap \cF^G_P(M_2,V_1 \cap V_2)$$
 $$=\cF^G_P(M_1 \oplus_M M_2,V_1\cap V_2)$$
\end{proof}

\begin{lemma}\label{suboject_simple}
 Let $M\in \cO^\frp_\alg$ be a simple object such that $\frp$ is maximal for $M$  and let $V$ be a smooth $L_P$-representation. Then any subquotient
 of $\cF^G_P(M,V)$ has the shape $\cF^G_P(M,W)$ for some smooth subquotient $W$ of $V.$
\end{lemma}

\begin{proof}
By the bi-exactness of $\cF^G_P$ it suffices to prove this statement for subobjects.
 Let $U\sub \cF^G_P(M,V)$ be a subobject.
 We recall a construction of \cite[Thm. 5.8]{OS2} which uses the simplicity of $M$. Set
 $U_{sm}=\varinjlim_H \Hom_H(\cF^G_P(M)|_H,U|_H)$ where the limit is over all compact open subgroups $H$  of $G.$
  It is proved that $U_{sm}$ is a subrepresentation of $\cF^G_P(M,V)_{sm}$ and that the latter object identifies with
  the smooth induction $i^G_P(V)=\ind^G_P(V)$ (for $V$ irreducible, but this holds also true in this general
 setting). Moreover, the natural map $\cF^G_P(M) \otimes \ind^G_P(V) \to \cF^G_P(M,V)$ is surjective giving rise by the very definition
 of this map to a  surjection $\phi:\cF^G_P(M) \otimes U_{sm} \to U$. Considering $U_{sm}$ as a subrepresentation of $i^G_P(V)$ as above we set  $W:=\{f(1)\mid f\in U_{sm}\}$. This
 is a smooth $L_P$-representation and  the map $\phi$ factorizes over $\cF^G_P(M,W).$ It follows that the image of
 the map $\phi$ coincides with $\cF^G_P(M,W).$ Hence $U=\cF^G_P(M,W).$
\end{proof}

\smallskip
The next statement is clear by the Jordan-H\"older principle for representation of the shape
$\cF^G_P(M,V)$ where $V$ is admissible smooth.

   \begin{prop}\label{extension}
    Every object $U$ in $\overline{^\infty\cF^P_\alg}$ is a successive extension of objects of the shape $\cF^G_Q(N,W)$ with $P\sub Q.$
   \end{prop}

\begin{proof}
As the direct sum of two objects of the kind $\cF^G_{Q_i}(M_i,V_i)$, $i=1,2,$ is contained in such an object we may 
suppose that $U$ is  some subquotient of $\cF^G_P(M,V)$.
Indeed we have $\cF^G_{Q_i}(M_i,V_i)\sub \cF^G_P(M,(V_i)_{\mid P})$ so that it suffices to treat the case $Q_1=Q_2=P$. But then we see that
$\cF^G_{P}(M_1,V_1) \oplus \cF^G_P(M_2,V_2)\sub \cF^G_P(M_1 \oplus M_2,V_1\oplus V_2).$

The proof is by induction on the length on $M$. If $M$ is simple (where we may assume that  $P$ is maximal for $M$ 
by the PQ-formula) then the statement follows from the above lemma.
Otherwise, let $M_1\sub M$ be some proper submodule and consider the  exact sequence
$$0 \to \cF^G_P(M/M_1,V) \to \cF^G_P(M,V) \stackrel{p}{\to} \cF^G_P(M_1,V) \to 0 .$$
So let $U=U_1/U_2$ be some subquotient of $\cF^G_P(M,V).$ We consider the induced exact sequence
$$0 \to \cF^G_P(M/M_1,V) \cap U_1/\cF^G_P(M/M_1,V) \cap U_2 \to U_1/U_2 \to p(U_1)/p(U_2) \to 0 .$$

If $\cF^G_P(M/M_1,V) \cap U_1 /\cF^G_P(M/M_1,V) \cap U_2 \in \{(0),U_1/U_2\}$ we may apply induction hypothesis to prove the claim.
But also in the other case the inductive hypothesis applies.

\end{proof}

\begin{prop}\label{contain}
Let $U\in \overline{^\infty\cF^P_\alg} $ and suppose that there exist $M\in \cO^\frp_\alg$ and $V \in \Rep^{\infty,\infty}_K(L_P)$ such that $U\subset \cF^G_P(M,V).$ Suppose further that $M$ is minimal, i.e. there is no proper quotient $N$ of $M$ such that $U\subset \cF^G_P(N,V).$ Then we have $\cG^G_P(U)=\bigoplus_{L} L \otimes W_L'$ as $P$-representations  where $L$ goes through all simple constituents
of $M$ and $W_L$ is a quotient of $i^{Q_L}_P(V)$  for some parabolic subgroup $Q_L$ depending on $L.$ 
\end{prop}

\begin{proof}
 The proof is by induction on the length of $M.$ If $M$ is simple
 then we use Lemma \ref{suboject_simple}  and  the claim follows from Proposition \ref{H0_uk}. In general let $M_1 \subset M$ be a proper submodule such that $M/M_1$ is simple. We consider as above the induced exact sequence
 $$0 \to \cF^G_P(M/M_1,V) \to \cF^G_P(M,V) \stackrel{p}{\to} \cF^G_P(M_1,V) \to 0 .$$
 If $p(U)=0$ we get a contradiction to the minimality of $M.$ If $U \cap  
 \cF^G_P(M/M_1,V) =0$, then $U\cong p(U)\subset \cF^G_P(M_1,V).$ By the induction hypothesis we see that $M_1$ appears in $\cG^G_P(U)$ as described above. 
 By considering the composition
 $U \hookrightarrow \cF^G_P(M,V) \to \cF^G_P(M_1,V)$ we get by applying of 
 $\cG^G_P$ a splitting  of the inclusion $M_1 \hookrightarrow M$ and thus a surjection $M \to M_1$ giving rise to a contradiction by the minimality of $M.$ 
 
 Hence we obtain a non-trivial exact sequence
 $$0 \to \cF^G_P(M/M_1,V) \cap U \to U \to p(U) \to 0 .$$
Thus the claim follows by induction once we have proved that
 $\cF^G_P(M_1,V)$ satisfies again the minimality condition with respect to $p(U).$ 
 But  if $N_1\subset M_1$ is a proper submodule with $p(U) \subset \cF^G_P(M_1/N_1,V)$ then it would follow that $U\subset \cF^G_P(M/N_1,V)$. This is again a contradiction to the minimality of $M.$ 
 \end{proof}


 \begin{prop}\label{injective}
  Let $M\in \cO^\frp_\alg$ be projective (resp. injective) and let $V\in \Rep^{\infty, \infty}_K(L_P) $ be an injective (resp. projective) object.
  Then $\cF^G_P(M,V)$ is  injective (resp. projective)
  in the category $\overline{^\infty\cF^P_\alg}$.
 \end{prop}

 \begin{proof} 
 We consider here the case of injective objects. The case of projective objects is treated in a dual way.  We consider 
 thus a monomorphism $Z_1 \hookrightarrow Z_2$ in our category $\overline{^\infty\cF^P_\alg}$ together
 with a morphism $Z_1 \to \cF^G_P(M,V)$. Since any object in $\overline{^\infty\cF^P_\alg}$ is a subquotient
 of an object lying in the image of a functor $\cF^G_Q$ with $P\subset Q$ we may suppose by enlarging $Z_2$ that
  is has for simplicity the shape $\cF^G_Q(N,W)$. Indeed if $Z_2$ is a submodule of $\cF^G_Q(N,W)$ this is clear.
  If on the other hand, $Z_2$ is a quotient of $\cF^G_Q(N,W)$ then we consider the preimage 
  $\tilde{Z_1}\hookrightarrow \tilde{Z_2}$ of
  $Z_1 \hookrightarrow Z_2$ in $\cF^G_Q(N,W)$. We get an induced map $f:\tilde{Z_1} \to \cF^G_P(M,V)$ and if this extends to $\tilde{Z_2}$
  then also to $Z_2$ since $\ker (\tilde{Z_1} \to Z_1)= \ker( \tilde{Z_2} \to Z_2)$ is mapped to zero under $f.$ 
    By the PQ-formula we see that
  $\cF^G_Q(N,W) \hookrightarrow \cF^G_Q(N,i^Q_P(W|L_P))=\cF^G_P(N,W|L_P).$ Hence we may even suppose that $P=Q.$
  Thus we arrived at the situation where we assume that $Z_2=\cF^G_P(N_2,W_2)$ for $N_2 \in \cO^\frp_\alg$ and $W_2 \in \Rep_K^{\infty,\infty}(G).$
  
  On the other hand, we may also suppose that  $Z_1$ has also the shape $\cF^G_Q(N,W).$ Indeed, by dividing out the kernel of the morphism  $Z_1 \to \cF^G_P(M,V)$  (from the very beginning) we may assume that it is injective as well.  By 
using Lemma \ref{intersection} we see that there are $N \in \cO^\frq_\alg$ and $W \in \Rep^{\infty,\infty}_K(L_Q)$
such that $\cF^G_Q(N,W) \subset \cF^G_Q(N_2,W_2)$ is a minimal object containing $Z_1.$ By Proposition \ref{contain} we deduce that $N$ and $W$ appear in $\cG^G_Q(Z_1)$. Hence the morphism $Z_1 \to \cF^G_P(M,V)$
extends automatically to a morphism $ \cF^G_Q(N,W) \to \cF^G_P(M,V)$.
 
 Hence we may think that our embedding  $Z_1 \hookrightarrow Z_2$ is of the shape 
 $\cF^G_Q(N_1,W_1) \hookrightarrow \cF^G_P(N_2,W_2)$. It follows by Proposition \ref{remark_conj_2}, the bi-exactness
 of $\cF^G_P$ and the exactness of the induction functor for smooth representations that it is
 induced by a surjection $N_2 \to N_1$ and a monomorphism $(W_1)_{\mid P} \hookrightarrow W_2.$
For this note that $\cF^G_Q(N_1,W_1) = \cF^G_Q(N_1) \hat{\otimes} W_1$ and  $\cF^G_P(N_2,W_2) = \cF^G_Q(N_2) \hat{\otimes} W_2$  . 

 So for proving that $\cF^G_P(M,V)$ is injective let $\cF^G_Q(N_1,W_1) \to \cF^G_P(M,V)$ be any morphism. Again it corresponds to a tuple of
 morphisms $M \twoheadrightarrow N_1$ and $(W_1)_{\mid P}  \hookrightarrow V.$
 Since $V$ is injective 
 we see that there
 is an extension $W_2 \to V.$ Further as $M$ is projective we have a lift
 $M \to N_2.$ The claim follows easily.
 \end{proof}

 \vskip8pt
 \begin{cor}
 The category $\overline{^\infty\cF^P_\alg}$ has enough injective and projective objects.
\end{cor}

\begin{proof}
As above we  consider here only the case of injectives. 
Let $U\in\overline{^\infty\cF^P_\alg}$. Suppose first that it has the shape $\cF^G_P(M,V).$
We choose a projective cover $N$
 of $M$ and an embedding $V \hookrightarrow W$ into a smooth injective $L_P$-representation $W$ of countable dimension. Then
 we have a topological embedding $\cF^G_P(M,V) \hookrightarrow \cF^G_P(N,W)$ and by
 the result above the object $\cF^G_P(N,W)$ is injective.
 
 In general we know by Proposition \ref{extension} that it is a successive extension of such objects.
 As such it has an injective envelope, as well (Indeed, suppose that $0 \to A_1 \to U \to A_2\to 0$ is exact and
 that $A_i \to I_i,i=1,2$ are monomorphism into injective objects. Then we get an exact sequence
 $0 \to I_1 \to I_1  \oplus_{A_1} U \to A_2\to 0$ and the middle term is isomorphic to $I_1 \oplus A_2$ by injectivity of 
 $I_1.$ But $I_1 \oplus A_2$ embeds into the injective object $I_1 \oplus I_2$. Therefore $U$ embeds into $I_1 \oplus I_2,$ as well.).
\end{proof}

\subsection{Extensions of generalized locally analytic Steinberg representations}

For a parabolic subgroup $P\sub G$, we abbreviate $I^G_P:=\Ind^G_P({\bf 1})$
and denote by $i^G_P$ the subspace of smooth vectors.
The attached Steinberg representation is given by the quotient 
$V^G_P= \Ind^G_P(\bf 1) / \sum_{Q \supsetneq P} \Ind^G_Q(\bf 1).$ 
We shall determine the Ext-groups of these objects in our compactified categories.

\vskip8pt
We recall a result from \cite{O1}. Here we denote by $^\infty\Ext^\ast$ the corresponding Ext-groups in the category
of smooth representations.

\begin{prop}\label{recall}
Let $I\sub \Delta.$ Then we have
$$^\infty\Ext^\ast_{L_I}({\bf 1}, {\bf 1})=\Lambda^\ast (X^\ast({\bf L}_I)).$$ \qed
\end{prop}

The next statement is contained in \cite[Thm. 9.8]{H2}.

\begin{lemma}\label{projective}
 For a parabolic subgroup $Q$ of $G$, let $M=M_Q(0)=U(\frg) \otimes_{U(\frq)} K$ be the generalized Verma module
 with respect to the trivial $Q$-module. Then $M$ is projective in $\cO^\frq_\alg.$ \qed
\end{lemma}

\begin{prop}\label{Ext_principal} Let $G$ be semi-simple and let $I,J \sub \Delta.$ Then we have
$$Ext^\ast_{^\infty\overline{\cF^B_\alg}}(I^G_{P_I},I^G_{P_J}) =\left\{
  \begin{array}{r@{\quad:\quad}l}\Lambda^\ast ( X^\ast({\bf L}_J))  & \mbox{ if } J\sub I \\ 0 & \mbox{ otherwise } \end{array} \right. .$$
\end{prop}

\begin{proof}
We set $P=P_I$ and $Q=P_J.$ 

 {\it 1. Case}. Suppose that $J \not\sub I.$ 
 Let $I^\bullet$ be an injective resolution  of the trivial $L_Q$-representation in the category $\Rep_K^{\rm \infty,\infty}(L_Q)$.
 Then by Lemma \ref{projective} and Proposition \ref{injective},  $\cF^G_Q(M_Q(0),I^\bullet)$ is an injective resolution of 
$I^G_Q.$
Let $J^\bullet$ be an injective resolution of the trivial $T$-representation in the category $\Rep_K^{\rm \infty,\infty}(T)$. 
Then $i^Q_B(J^\bullet)$ is an injective resolution of $i^Q_B$ (in the category of smooth representations) since the induction functor
is exact and has with the Jacquet functor an exact left adjoint.
Hence the embedding ${\bf 1}_Q \to i^Q_B$
extends to a morphism of complexes $I^\bullet \to i^Q_B(J^\bullet)$. Here we may suppose by standard arguments  that 
the maps in each degree  are injective.
We consider
the induced (injective) maps $\cF^G_Q(M_Q(0),I^\bullet) \to \cF^G_Q(M_Q(0),i^Q_B(J^\bullet))=\cF^G_B(M_Q(0),J^\bullet).$
We shall see that any map $I^G_P \to \cF^G_B(M_Q(0),J^i), i\geq 0,$ vanishes which is enough for our claim. Indeed by Remark \ref{PQ_geht_auch}
it is induced on the Lie algebra part by a map $M_Q(0) \to M_P(0)$. Any such map vanishes if $Q \not\sub P.$

\vskip8pt
{\it 2. Case}. Suppose that $J \sub I.$ 
  Then by applying Frobenius reciprocity any map $I^G_P \to \cF^G_Q(M_Q(0),I^i)=\Ind^G_Q(I^i)$ is given by
 a map $(I^G_P)_{U_Q}=H^0(U_Q, I^G_P)' \to I^i.$ The left hand side coincides by Proposition  \ref{h0_verma} with
 $H^0(\fru_Q,M_P(0))'$ which is a sum of algebraic representations and which
 contains the trivial representation. Since any map between an algebraic representation different
 from the trivial one and a smooth representation vanishes we see that 
 any map $(I^G_P)_{U_Q} \to I^i$ corresponds to a map ${\bf 1} \to I^i$. Hence  the series of maps determines 
 $^\infty\Ext^\ast_{L_J}({\bf 1}, {\bf 1})$ which coincides with $\Lambda^\ast ( X^\ast({\bf L}_J))$ by 
 Proposition \ref{recall}.
\end{proof}

\begin{thm} Let $G$ be semi-simple. {\it Let $I,J \sub \Delta$. Then

$$\Ext_{^\infty\overline{\cF^B}}^i(V^G_{P_I}, V^G_{P_J})=\left\{ \begin{array}{cc}
                                   K & | I \cup J \setminus I\cap J| =i \\
                                  (0) & \mbox{otherwise}
                                \end{array} \right. .
 $$}
 \end{thm}

\begin{proof}
In \cite{OSch} we proved that the following complex is an acyclic resolution of $V^G_{P_I}$ by locally analytic $G$-representations,
\begin{equation}
0 \rightarrow  I^G_G \rightarrow \bigoplus_{I\sub K \sub \Delta \atop |\Delta\setminus K|=1}I^G_{P_K} \rightarrow \bigoplus_{I\sub K \sub \Delta \atop |\Delta\setminus K|=2}I^G_{P_K}
\rightarrow \dots \rightarrow \bigoplus_{I\sub K \sub \Delta
  \atop |K\setminus I|=1}I^G_{P_K}\rightarrow I^G_{P_I} \rightarrow V^G_{P_I}\rightarrow 0.
\end{equation}
The smooth version of this complex was used in \cite{O1} together with the smooth version of Proposition \ref{Ext_principal}
to get by formal arguments the smooth version of our theorem. Hence
 the rest of the proof is the same as in loc.cit.
\end{proof}

If $G$ is not necessarily semi-simple, then we have as in the smooth case a contribution of the center $Z(G)$.
By using a Hochschild-Serre argument (cf. loc.cit.) we conclude:
\begin{cor}
Let $G$ be reductive with center $Z(G)$ of rank $d.$  Let $I,J \sub \Delta.$ Then we have
$$Ext_{^\infty\overline{\cF^B}}^i(V^G_{P_I},V^G_{P_J})=\left\{ \begin{array}{r@{\quad:\quad}l}
  K^{d \choose j }  & i = |I\cup J| -  |I\cap J| +j,\; j=0,\ldots,d\\ 0 & otherwise \end{array} \right. $$
\end{cor}\qed

\vspace{0.5cm}

\subsection{Adjunction}

As a last application we want to discuss some  adjunction formulas. For this we need some preparations.

\begin{lemma}\label{Lemma_cdot}
 Let $x,w\in W$ and let $\chi:T \to K^\ast$ be an  algebraic character. Then
 $$(x \cdot_{B} \chi)^{w}={\rm Ad}(w)(x)\cdot_{B^{w^-1}} \chi^{w}.$$
\end{lemma}

\begin{proof}
We compute
\begin{eqnarray*}
          (x \cdot_{B} \chi)^{w} & =  & w(x(\chi+\rho_{B})-\rho_{B}) \\
          & = & {{\rm Ad}(w)(x)}(w(\chi +\rho_{B})-w\rho_{B})\\
          & = & {{\rm Ad}(w)(x)}((\chi^{w} + \rho_{B^{w^{-1}}})-\rho_{B^{w^{-1}}}) \\
          & = & {{\rm Ad}(w)(x)}\cdot_{B^{w^{-1}}} \chi^{w} \;\; .
         \end{eqnarray*}
         \end{proof}

Let $M=M_{\overline{B}}(\chi)\in \cO^{\overline{\frb}}_\alg$ be a Verma module with respect to the opposite Borel subgroup $\overline{B}.$ Then 
$\cF^G_{\overline{B}}(M)=\Ind^G_{\overline{B}}(\chi^{-1})$ and 
$\overline{H}_0(U_B,\cF^G_{\overline{B}}(M))=(H^0(\fru_{\overline{B}},M)')^{w_0}$
by Lemma \ref{h0_verma_alle}. Hence there is a natural homomorphism 
$(\chi^{-1})^{w_0} \to \overline{H}_0(U_B,\cF^G_{\overline{B}}(M)).$ If further $\chi$ is $\overline{B}$-dominant, then we have
moreover a natural homomorphism 
$$((w_0\cdot_{\overline{B}} \chi)^{-1})^{w_0} \to \overline{H}_0(U_B,\cF^G_{\overline{B}}(M)).$$
These maps lead by composing with the functor $V \mapsto V_{U_B}=\overline{H}_0(U_B,V)$ to the following statements.

\begin{thm}
  Let $\chi$ be a  $\overline{B}$-dominant algebraic character.
 Then for any $w\in W$ and any highest weight module $M\in \cO^{\frb^w}_\alg$ one has the identity
 $$\Hom_G(\Ind^G_{\overline{B}}(\chi^{-1}),\cF^G_{B^w}(M))=
 \Hom_T(((w_0\cdot_{\overline{B}}\chi)^{-1})^{w_0},\cF^G_{B^w}(M)_{U_B})$$
\end{thm}

\begin{proof}

\vskip8pt
i) First let $M=M_{B^w}(\lambda)$ be a Verma module for some algebraic character $\lambda$ of $T.$
We start with the observation that both sides are at most one-dimensional. Indeed as for the LHS this
follows from Proposition \ref{Hom_twisted_Verma}. 
As for the RHS we can identify it (see below) with the anti-dominant eigenspace in  $H^0(\fru_{B^w},M)$.
This eigenspace is one-dimensional, as well.

\smallskip
Now we check, that the LHS does not vanish iff the RHS does.
Since $\chi$ is $\overline{B}$-dominant we see
that $\chi^{w^{-1}w_0}$ is $B^w$-dominant. The LHS does not vanish by Proposition 
\ref{Hom_twisted_Verma} if and only if 
$\lambda^{w_0w}  \uparrow_{\overline{B}} \chi.$
This is equivalent to $\lambda \uparrow_{B^w} \chi^{w^{-1}w_0}$ by Lemma \ref{Lemma_cdot}.
Since 
$\chi^{w^{-1}w_0}$ is $B^w$-dominant and $w^{-1}w_0w$ is the longest Weyl group element in $W$ with respect to
$B^w$, we see that the latter condition is equivalent to 
$(w^{-1}w_0w) \cdot_{B^w} (\chi^{w^{-1}w_0}) \uparrow_{B^w} \lambda$.

On the other hand, the Jacquet module $\Ind^G_{B^w}(\lambda^{-1})_U$ coincides by Proposition \ref{h0_verma_alle} with $(H^0(\fru_{B^w},M)')^w.$ Its weights are given by the characters $(\mu^{-1})^{w}$ with $\mu \uparrow_{B^w} \lambda.$
Moreover, $(w_0\cdot_{\overline{B}} \chi)^{w_0}= ww^{-1}w_0(w_0\cdot_{\overline{B}} \chi)=
w(w^{-1}w_0w\cdot_{B^w} \chi^{w^{-1}w_0})$ by Lemma \ref{Lemma_cdot}. Thus
the RHS does not vanish iff the LHS does not vanish. 
To see that the natural map between these one-dimensional spaces is an isomorphism follows in principal from Theorem \ref{fullyfaithful} .

\vskip8pt
ii) Now let $M$ be a quotient of $M_{B^w}(\lambda).$ Then $\cF^G_{B^w}(M)\sub \Ind^G_{B^w}(\lambda^{-1})$ so that
both vector spaces in the above stated formula are at most one-dimensional. Moreover, we have a commutative diagram

\begin{eqnarray*}
\Hom_G(\Ind^G_{\overline{B}}(\chi^{-1}),\Ind^G_{B^w}(\lambda^{-1})) & =  &
 \Hom_T(((w_0\cdot_{\overline{B}}\chi)^{-1})^{w_0},\Ind^G_{B^w}(\lambda^{-1})_{U_B}) \\
  \uparrow & & \uparrow \\  
  \Hom_G(\Ind^G_{\overline{B}}(\chi^{-1}),\cF^G_{B^w}(M)) & \to &
 \Hom_T(((w_0\cdot_{\overline{B}}\chi)^{-1})^{w_0},\cF^G_{B^w}(M)_{U_B})
\end{eqnarray*}
The upper line is an isomorphism by the first case. The LHS is an injection. In particular the lower line
is an injection, as well. Since the spaces in question are at most one-dimensional the statement follows
easily in this case. Note that if $M$ is a proper quotient of $M_{B^w}(\lambda)$, then the objects
in the lower line vanishes and the claim is trivial. 

\end{proof}

\begin{rmk}
 In \cite{Br} and \cite{BC} are presented adjunction formulas which use on the RHS Emerton Jacquet functor
 and which have a different style.
\end{rmk}

\bibliographystyle{plain}
\bibliography{functorial_new_bib}

\end{document}